\documentclass[final, 11pt]{article}

\usepackage{mdwlist}
\usepackage{graphicx}
\usepackage{amsfonts}
\usepackage{ulem}
\usepackage{amsmath}
\usepackage{amsthm}
\usepackage{amssymb}
\usepackage{fancyhdr}
\usepackage{titlesec}
\usepackage{indentfirst}
\usepackage{booktabs}
\usepackage{verbatim}
\usepackage{color}
\usepackage{latexsym}
\usepackage{amscd}
\usepackage{esvect}
\usepackage{graphicx}
\usepackage{upgreek}
\usepackage{subcaption}   
\usepackage{caption}
\usepackage[page,header]{appendix}
\usepackage{titletoc}
\usepackage{mathrsfs}
\usepackage[numbers]{natbib}
\usepackage{float}
\usepackage{lipsum}
\usepackage[export]{adjustbox} 
\usepackage{stackengine} 
\numberwithin{equation}{section}
\newtheorem{theorem}{Theorem}[section]
\newtheorem{lemma}[theorem]{Lemma}

\newtheorem{proposition}[theorem]{Proposition}

\newtheorem{remark}[theorem]{Remark}

\allowdisplaybreaks

\makeatletter
\newcommand{\ostar}{\mathbin{\mathpalette\make@circled\star}}
\newcommand{\make@circled}[2]{%
  \ooalign{$\m@th#1\smallbigcirc{#1}$\cr\hidewidth$\m@th#1#2$\hidewidth\cr}%
}
\newcommand{\smallbigcirc}[1]{%
  \vcenter{\hbox{\scalebox{0.77778}{$\m@th#1\bigcirc$}}}%
}
\makeatother

\def\spnu{\langle v\rangle^k \boldsymbol{\nu}}
\def\sp{\langle v\rangle^k}

\def\DT{^{(\delta)}}

\def\pn{\partial^n}
\def\pl{\partial^l}
\def\bff{\bf {f}}
\def\bfg{\bf {g}}
\def\nk{\binom{n}{k}}
\def\nl{\binom{n}{l}}
\def\nm{\binom{n}{m}}
\def\lm{\binom{l}{m}}

\begin{document}

\title{On the multi-species Boltzmann equation with uncertainty and its stochastic Galerkin approximation \footnote{E. S. Daus acknowledges partial support from the Austrian Science Fund (FWF), grants P27352 and P30000, S. Jin is supported by NSFC grants No. 11871297 and No. 31571071, L. Liu is supported by the start-up fund from The Chinese University of Hong Kong. }}
\author{Esther S. Daus\footnote{
Institute for Analysis and Scientific Computing, Vienna University of Technology, Wiedner Hauptstra{\ss}e 8-10, 1040 Wien, Austria (esther.daus@tuwien.ac.at)},
Shi Jin\footnote{School of Mathematical Sciences, Institute of Natural Sciences, MOE-LSC and SHL-MAC, Shanghai Jiao Tong University, Shanghai, China (shijin-m@sjtu.edu.cn)}, 
Liu Liu\footnote{Department of Mathematics, The Chinese University of Hong Kong, Shatin, N.T., Hong Kong SAR 
(lliu@math.cuhk.edu.hk)}}
\date{\today}
\maketitle

\abstract{In this paper the nonlinear multi-species Boltzmann equation with random uncertainty
coming from the initial data and collision kernel is studied. Well-posedness and long-time behavior -- exponential decay to the global equilibrium
-- of the analytical solution, and spectral gap estimate for the corresponding linearized gPC-based stochastic Galerkin system are obtained, by using and extending the analytical tools provided in [M. Briant and E. S. Daus, {\it Arch. Ration. Mech. Anal.}, 3, 1367--1443, 2016]  for the deterministic
problem in the perturbative regime, and in [E. S. Daus, S. Jin and L. Liu, {\it Kinet. Relat. Models}, 12, 909--922, 2019] for the single-species problem with uncertainty. \textcolor{black}{The well-posedness result of the sensitivity system presented here has not been obtained so far neither in the single species case nor in the multi-species case.} }

\section{Introduction}

We consider the multi-species Boltzmann equation describing the evolution of a multi-species mono-atomic nonreactive gaseous mixture with additional uncertainty coming from the initial data and collision kernel, which 
was studied analytically in the deterministic setting in \cite{BBBD,BBBG,BGPS,BS,Marc16,BD16,Linearized-Boltz16}. Compared to the single-species deterministic analysis of the Boltzmann equation, dealing with different conserved quantities \textcolor{black}{due to different thermodynamic properties of mixtures (see the multi-species H-theorem in \cite{Des05, Multi-99}) provided the main difficulty in the analysis for the multi-species deterministic problem. For more details see subsection \ref{subsec:2.2}.}

In this paper, we deal with the multi-species Boltzmann equation with an additional random parameter described by the random variable $z$, which lies in the random space $I_z$ with a probability measure $\pi(z)dz$. Thus, the solution $f=f(t,x,v,z)$ depends also on the random parameter $z \in I_z$.
We will conduct the sensitivity analysis, which aims to study how the random inputs in the system propagate in time and how they influence the solution in the long time \cite{SmithBook}. To our knowledge, uncertainty quantification (UQ) for any nonlinear multi-species kinetic model has not been studied so far, 
while general single-species linear and non-linear collisional kinetic problems with multiple scales and uncertainty were studied in \cite{LJ18}. 

Research on uncertainty quantification for kinetic equations has not started until recently, and the reason for the growing interest in these problems is the following. 
Kinetic equations, derived from $N$-body Newton's equations via the mean-field limit \cite{BGP}, 
typically contain an integral operator modeling interactions between particles. 
Since calculating the collision kernel from first principles is impossible for complex particle systems, only empirical formulas are used for general particles \cite{Cercignani}. Consequently, this inevitably brings modeling errors, so the collision kernel contains some uncertainty.  
Other sources of uncertainties may come from inaccurate measurements of the initial or boundary data, 
forcing or source terms. 
We refer to the book \cite{UQ-Book} and
the recent articles and reviews \cite{DesPer, DPZ, HuReview, Ma, Jin-Liu, Qin-Li, Liu-BP, LJ18, LM19} for more detailed studies in this direction. 

The main goal of this paper is to study the well-posedness and long-time behavior of the nonlinear multi-species Boltzmann equation 
under the impact of random uncertainty and its stochastic Galerkin approximation in the perturbative regime. 
The first part of our paper (Section \ref{sec:3}) studies the well-posedness and exponential decay of the solution with random initial data 
and collision kernel in suitable Sobolev spaces in the perturbative setting, in which the initial data is assumed to be close to the global equilibrium. 
Our proof is based on the analysis of the Cauchy theory of the multi-species Boltzmann equation with uncertainty in the weighted Lebesgue space $L^1_vL^{\infty}_x(\langle v \rangle^k)L^{\infty}_z$ \textcolor{black}{(see \eqref{function.spaces} for the precise definitions)} with a polynomial weight of order $k > k_0$ (where $k_0$ is the threshold derived in \cite[Section 6]{BD16}, which recovers in the particular case of a multi-species hard spheres mixture (with equal molar masses) the optimal threshold \textcolor{black}{of finite energy} $k_0=2$ obtained in the single-species setting in \cite{GMM}). 

The additional difficulty in our framework with uncertainty compared to the deterministic setting is to handle the extra high-order derivatives in the random parameter $z$, which naturally appear from the fact that we introduce uncertainty into the model. 
We refer to the equations obtained by taking the $z$-derivatives of \textcolor{black}{the $i$-th component of} the density functions governed by the multispecies Boltzmann equation as the {\it sensitivity equations}. We manage to control these new terms containing high-order $z$-derivatives by designing a {\it new} decomposition built upon the factorization of Gualdani, Mischler and Mouhot in \cite{GMM}, with a mathematical induction in the order of $z$-derivatives. This factorization technique was established by Gualdani, Mischler and Mouhot in \cite{GMM}, later adapted to the nonlinear perturbative setting in \cite{Marc}, and generalized to the multi-species deterministic framework \textcolor{black}{with different molar masses} in \cite{BD16}. \textcolor{black}{For more details on the factorization method see section \ref{subsec:1}.} We want to emphasize that there has not been established any rigorous existence analysis for uncertain kinetic equations in any previous work \cite{ZLM-DG, DJL, Jin-Liu, Liu-BP, LJ18} yet, even not for the single-species case.

\textcolor{black}{Concerning the task of numerically solving kinetic equations with uncertainties}, one of the standard and efficient numerical methods is the generalized polynomial chaos approach in the stochastic Galerkin (referred to as gPC-SG) framework \cite{Ghanem, GWZ, Xiu}. Compared to the classical Monte Carlo method, the gPC-SG approach enjoys a spectral accuracy in the random space--if the solution is sufficiently smooth--while the Monte Carlo method converges with the rate of $O(1/\sqrt{N})$, where $N$ is the number of simulations. 
\textcolor{black}{Note that the smoothness of the solution in the random space is one motivation for us to use the SG method. However, other types of non-intrusive methods, such as the stochastic collocation method, could also work well especially for high-dimensional problems, but for us it seemed to be mathematically more interesting to study the sensitiveness of the Galerkin system and its convergence.}

The second part of our paper (Section \ref{sec:4}) obtains the spectral gap estimate for the linearized gPC-Galerkin system. 
Compared to \cite{DJL} on the single-species gPC-SG Boltzmann system, the generalization to the multi-species case here can be done by adapting techniques from the proof for the multi-species H-theorem, see for instance \cite{ Linearized-Boltz16, Des05}. Establishing this spectral estimate is essential in order to understanding the long-time behavior of the gPC-SG approximation.

{\textcolor{black}We remark that our work relies on several existing literature on UQ for general kinetic models \cite{UQ-Book}, sensitivity analysis \cite{LJ18}, spectral convergence of the gPC-Galerkin method \cite{DJL} and multi-species Boltzmann equations \cite{BD16}. Readers may refer to those work for a more detailed overview. 
}

The paper is organized as the following. In Section \ref{sec:2}, we introduce the multispecies Boltzmann equation with uncertainty and 
present \textcolor{black}{the assumptions} for the two main results of this paper. 
In Section \ref{sec:3}, we show the existence and uniqueness of the sensitivity equations in the perturbative setting and establish the 
exponential decay of each order $z$-derivative of the solution. 
In Section \ref{sec:4}, we extend the previous work \cite{DJL, LJ18} to the multi-species setting and obtain the 
spectral gap for the linearized gPC-SG system. Finally, we formulate our conclusions in Section \ref{sec:5}. 
 
\section{The multispecies Boltzmann equations with uncertainty} 
\label{sec:2}

The evolution of a dilute ideal gas composed of $N \geq 2$ different species of chemically non-interacting mono-atomic particles
with same molar particle masses can be modeled by the following system of Boltzmann equations (see \cite{Marc16, BD16,Linearized-Boltz16} for the deterministic case), 
with some uncertainty characterized by a random variable $z\in I_z$, coming from both the initial data and the collision kernels, 
\begin{align}
\label{model}
\begin{split}
&\displaystyle\partial_t F_i + v\cdot\nabla_x F_i = Q_i (F), \qquad t>0, \\[4pt]
&\displaystyle F_i(0,x,v,z)=F_{I,i}(x,v,z), \qquad 1\leq i\leq N, \, (x,v)\in \mathbb T^3 \times \mathbb R^3, \, z \in I_z, 
\end{split}
\end{align}
where ${\bf F}=(F_1, \cdots, F_N)$ is the distribution function of the system, 
with $F_i$ ($1\leq i \leq N$) describing \textcolor{black}{the distribution function} of the $i$-th species. The spatial domain $\mathbb T^3$ is the three-dimensional torus.
For the sake of simplicity of the presentation, compared to \cite{BD16}, we set all the molar masses to be equal, e.g., $m_i=1$, for $i=1, \cdots, N$. 
The right-hand side of the kinetic equation (\ref{model}) is the $i$-th component of the nonlinear collision operator 
${\bf Q}({\bf F}) = (Q_1({\bf F}), \cdots, Q_N({\bf F}))$, and is defined by 
\begin{equation}\label{Q-F} Q_i({\bf F}) = \sum_{j=1}^N Q_{ij}(F_i, F_j), \qquad 1\leq i\leq N, \end{equation}
where $Q_{ij}$ models interactions between particles of species $i$ and $j$ ($1\leq i, j \leq N$), 
\begin{equation}\label{Q-F.1} Q_{ij}(F_i, F_j)(v,z) = \int_{\mathbb R^3\times\mathbb S^2} B_{ij}(|v-v^{\ast}|, \cos\theta, z)(F_i^{\prime}F_j^{\prime\ast}
 - F_i F_j^{\ast})\, dv^{\ast}d\sigma, 
 \end{equation}
 where we used the shorthands \(F_{i}^{\prime}=F_{i}\left(v^{\prime}\right), F_{i}=F_{i}(v), F_{j}^{*}=F_{j}\left(v_{*}^{\prime}\right)\) and \(F_{j}^{*}=F_{j}\left(v_{*}\right)\). \textcolor{black}{The velocities before and after the collisions are described by the following relation: 
 \begin{equation*}
  v' = \frac{v+v^*}{2} + \frac{|v-v^*|}{2}\sigma, \quad
	v'^* = \frac{v+v^*}{2} - \frac{|v-v^*|}{2}\sigma,
\end{equation*}
which follows from the fact that we assume the collisions to be
elastic, i.e., the momentum and kinetic energy are conserved on the microscopic level:
\begin{equation*}
  v' + v'^* = v + v^*, \quad \frac12|v'|^2 + \frac12|v'^*|^2 
	= \frac12|v|^2 + \frac12|v^*|^2.
\end{equation*}}
Here the collision kernel $B$ depends on the relative velocity $|v-v^{\ast}|$, the cosine of the deviation angle $\theta$, and the random variable 
$z \in I_z\subseteq \mathbb{R}$. For simplicity, we consider a one-dimensional random space, but our analysis can be easily extended to higher dimensional cases as well. 
 
The global equilibrium, which is the unique stationary solution to (\ref{model}), is given by 
$ M^{\infty}=( M_1^{\infty}, \cdots,  M_N^{\infty})$, with 
$$M_i^{\infty}(v) = c_{\infty, i} \left(\frac{1}{2\pi k_B\theta_{\infty}} \right)^{3/2}\exp\left(-\frac{|v-u_{\infty}|^2}{2 k_B \theta_{\infty}}\right), $$
where for $1\leq i \leq N$, 
\begin{align*}
&\displaystyle  c_{\infty,i} = \int_{\mathbb T^3\times\mathbb R^3} M_i^{\infty} \, dx dv, 
\qquad\qquad \rho_{\infty}=\sum_{i=1}^N c_{\infty, i},    \\[4pt]
&\displaystyle  u_{\infty} = \frac{1}{\rho_{\infty}}\sum_{i=1}^N \int_{\mathbb T^3\times\mathbb R^3} v \,  M_i^{\infty} \, dx dv,  \qquad
\theta_{\infty} = \frac{1}{3 \rho_{\infty}}\sum_{i=1}^N \int_{\mathbb T^3\times \mathbb R^3} |v-u_{\infty}|^2\,  M_i^{\infty}\, dxdv. 
\end{align*}
 By translating and scaling the coordinate system, one can assume $u_{\infty}=0$ and $k_B\theta_{\infty}=1$, and then the global equilibrium becomes
$${\bf M} = (M_i)_{1\leq i\leq N}, \qquad M_i(v) = c_{\infty,i}\left(\frac{1}{2\pi}\right)^{3/2}e^{-\frac{|v|^2}{2}}. $$ 
\hspace{1cm}

\subsection{Main assumptions on the random collision kernel}
We summarize here the assumptions on the random collision kernel that are needed throughout the whole paper: 
\renewcommand{\labelenumi}{(H\theenumi)}
\begin{enumerate}
\item The following symmetry holds for each $z \in I_z \subseteq \mathbb{R}$: 
\begin{equation}\label{b-Assump0} B_{ij}(|v-v_*|,\cos\theta, z) = B_{ji}(|v-v_*|,\cos\theta, z)\quad\mbox{for 
}1\le i,j\le N. \end{equation}
\item The collision kernels for each $z \in I_z \subseteq \mathbb{R}$ are decomposed into the product
\begin{equation}\label{b-Assump1} B_{ij}(|v-v_*|,\cos\theta, z) = \Phi_{ij}(|v-v_*|)\, b_{ij}(\cos\theta, z),
\quad 1\le i,j\le N,
\end{equation}
where the functions $\Phi_{ij}\ge 0$ are called the kinetic part and 
the angular part $b_{ij}(\cos\theta, z)>0$ is assumed to be uncertain. 

\item 
We consider the case of hard potentials $\gamma \in (0,1]$ or Maxwellian molecules ($\gamma=0$), and thus the kinetic part takes the form: 
$$\Phi_{ij}(|v-v_*|)=C_{ij}^{\Phi}\,|v-v_*|^{\gamma}, \quad 
C_{ij}^{\Phi}>0,~\:\gamma\in[0,1], \quad \forall\: 1 \leq i,j\leq N.$$
\item For the angular part, for each $z \in I_z \subseteq \mathbb{R}$ we assume a strong form of Grad's angular 
cutoff, i.e., $\exists$ 
$C_b$, $C_{b_1}>0$ such that
for all $1\le i,j\le N$ and $\theta\in[0,\pi]$,
\begin{equation}\label{bb}  0<b_{ij}(\cos\theta,z)\le C_b\, |\sin\theta|\,|\cos\theta| \leq C_b, \quad 
 \partial_{\theta} b_{ij}(\cos\theta,z)\le C_{b_1}. \end{equation}
Furthermore,
$$  \min_{1\le i\le 
N}\inf_{\sigma_1,\sigma_2\in\ \mathbb S^2}\int_{\mathbb S^2}\min\big\{ 
b_{ii}(\sigma_1\cdot\sigma_3,z),b_{ii}(\sigma_2\cdot\sigma_3,z)\big\}\:d\sigma_3 
 > 0. $$
 
\item In addition, we assume the following condition on $|\partial^k_z b_{ij}|$ for all $z$: 
 \begin{equation}\label{b-Assump} |\partial^k_z b_{ij}(\cos\theta, z)| \leq C_b, \qquad \forall \, 0\leq k \leq r, \quad 1\leq i, j \leq N, \end{equation}
 where $r \in \mathbb N$ is determined by the regularity of the random initial data, and $C_b$ is the same upper bound as in \eqref{bb}. 
\end{enumerate}

In (H1)--(H4), for each fixed $z$ the same conditions are assumed as in the deterministic problem \cite{BD16}. 
The new assumption appears in (H5).  We mention that our analysis in this work also applies to the case when the kinetic part $\Phi_{ij}$ of the 
collision kernel is assumed uncertain, i.e., $B_{ij}$ takes the form: $$B_{ij}(|v-v_*|,\cos\theta, z) = \Phi_{ij}(|v-v_*|,z)\,b_{ij}(\cos\theta). $$

\subsection{State of the art on the multi-species deterministic Boltzmann equation}\label{subsec:2.2}
As already mentioned above, the main difficulty of the deterministic multi-species Boltzmann equation compared to the single-species Boltzmann equation lies in the different conserved quantities: namely, the mass of each species is conserved, while for the momentum and kinetic energy only the sum of all the species is conserved, see \cite{Multi-99, Des05}. Because of this, the proof of an explicit spectral-gap estimate of the linearized single-species operator \cite{Mouhot1} had to be changed significantly in the multi-species framework in \cite{Linearized-Boltz16} by carefully exploiting these new collision invariants. The stability of this spectral-gap estimate around non-equilibrium Maxwellian distributions was studied in \cite{BBBG}. The full Cauchy theory for the 
inhomogeneous Boltzmann equation for mixtures in the perturbative regime was formulated without going to any higher order Sobolev regularity \cite{BD16}, by using the factorization method of \cite{GMM}. 
Besides this, in \cite{BD16} a new multi-species Carleman's representation and a new Povzner-type inequality was proved, due to the loss of symmetry arisen from different masses. 
In \cite{BGPS13, BS}, compactness of one part of the linearized multi-species operator was studied, moreover, in \cite{BGP} it was shown that in the diffusive limit, the multi-species Boltzmann equation converges to the Maxwell-Stefan system. In \cite{BBD18}, the Chapman-Enskog asymptotics for a mixture of gases was presented. 

Finally, we also want to mention the very recent work \cite{Gamba19} on the homogeneous multi-species Boltzmann system, for which it seems to be rather hard to conduct the sensitivity analysis and study the long-time behavior in the UQ setting, since \textcolor{black}{the logarithmic entropy functional cannot be evaluated for the $z$-derivatives of the distribution function, due to their lack of positivity. }

\section{Existence and exponential decay of the solution to the sensitivity system}
\label{sec:3}

This section will discuss the existence of a solution and the exponential decay to global equilibrium of the multi-species Boltzmann equation {\color{black}in the perturbative setting} with random initial data and collision kernel. In the following, we will introduce the same notation and we will use similar techniques as in \cite{BD16}, where the Cauchy theory for the (deterministic) multi-species Boltzmann system was studied. 
Using the ansatz
\begin{equation}\label{PS} F_i(t,x,v,z) = M_i(v) + f_i(t,x,v,z), \end{equation} the equation for ${\bff}=(f_1, \cdots, f_N)$ satisfying the perturbed multi-species Boltzmann equation reads as
\begin{equation}\label{Model} \partial_t {\bf f} + v\cdot \nabla_x {\bf f} = {\bf L}({\bf f}) + {\bf Q}({\bf f}),  \qquad {\bf f}(0,x,v,z)={\bf f_0}(x,v,z), 
\end{equation}
where ${\bf L}=(L_1, \cdots, L_N)$ is the linearized Boltzmann collision operator with its $i$-th ($1\leq i \leq N$) component given by 
$$ L_i({\bff}) = \sum_{j=1}^N L_{ij}(f_i, f_j), \qquad L_{ij}(f_i, f_j) = Q_{ij}(M_i, f_j) + Q_{ij}(f_i, M_j), $$
with $Q_{ij}(\cdot,\cdot)$ defined in \eqref{Q-F.1},
and the nonlinear Boltzmann collision operator ${\bf Q}= (Q_1, \cdots, Q_N)$ is defined in \eqref{Q-F} and \eqref{Q-F.1}.


\subsection{\textcolor{black}{Presentation and discussion} of the main result}
\label{subsec:1}

The proof of the main result of Section \ref{sec:3} uses techniques of \cite[Section 6]{BD16} which rely on the idea of \textcolor{black}{a} nonlinear version of the factorization method of \cite{GMM} presented in \cite{Marc}. 

We first briefly recall some propositions in \cite{BD16} to prepare us for the analysis. 
Define the truncation function $\Theta_{\delta}(v, v^{\ast}, \sigma)\in C^{\infty}(\mathbb R^3 \times \mathbb R^3)$ bounded by $1$ on the set
$$ \left\{ |v|\leq \delta^{-1} \, \text{   and   } 2\delta \leq |v-v^{\ast}| \leq \delta^{-1}  \,  \text{   and   }  |\cos\theta| \leq 1 - 2\delta \right\}, $$
and its support included in the set
$$ \left\{ |v|\leq 2\delta^{-1} \, \text{   and   } \delta \leq |v-v^{\ast}| \leq 2\delta^{-1}  \,  \text{   and   }  |\cos\theta| \leq 1 - \delta \right\}, $$
where $\delta \in (0,1)$ is to be chosen. 
Define the splitting of the linear operator \\ ${\bf G} = (G_1, \cdots, G_i, \cdots, G_N)$ as
\begin{equation}\label{splitting.G}
{\bf G} = {\bf L} - v\cdot\nabla_x = {\bf A}^{(\delta)} + {\bf B}^{(\delta)} - {\boldsymbol{\nu}} - v\cdot\nabla_x, \end{equation}
\textcolor{black}{where ${\boldsymbol{\nu}}=(\nu_1, \cdots, \nu_N)$ is a multiplicative operator called collision frequency, which also depends on the random variable $z$: 
$$\nu_i(v,z) = \sum_{j=1}^N \nu_{ij}(v,z), \qquad
\nu_{ij}(v,z) = C_{ij}^{\Phi}\int_{\mathbb R^3\times\mathbb S^2} b_{ij}(\cos\theta, z) |v-v^{\ast}|^{\gamma} M_i(v^{\ast})\, d\sigma dv^{\ast},$$ }
and the operators ${\bf A}^{(\delta)} = \left(A_i^{\delta}\right)_{1\leq i \leq N}$ and ${\bf B}^{(\delta)} = \left(B_i^{\delta}\right)_{1\leq i \leq N}$ are defined by
\begin{align}
\label{A-B}
\begin{split}
&\displaystyle A_i^{(\delta)}({\bf f}(v,z)) = \sum_{j=1}^N C_{ij}^{\Phi}\int_{\mathbb R^3\times\mathbb S^2} \Theta_{\delta}
(M_j^{\prime\ast}f_i^{\prime} + M_i^{\prime}f_j^{\prime\ast} - M_i f_j^{\ast}) b_{ij}(\cos\theta, z) |v-v^{\ast}|^{\gamma} d\sigma dv^{\ast}, \\[4pt]
&\displaystyle B_i^{(\delta)}({\bf f}(v,z))  = \sum_{j=1}^N C_{ij}^{\Phi}\int_{\mathbb R^3\times\mathbb S^2}  
(1 - \Theta_{\delta}) (M_j^{\prime\ast}f_i^{\prime} + M_i^{\prime}f_j^{\prime\ast} - M_i f_j^{\ast})
b_{ij}(\cos\theta, z) |v-v^{\ast}|^{\gamma} d\sigma dv^{\ast}. 
\end{split}
\end{align}
\textcolor{black}{The results in \cite{BD16}} have shown that ${\bf A}^{(\delta)}$ has some regularizing effects and that 
\begin{equation}{\bf G_1}\DT: = {\bf B}^{(\delta)} - {\boldsymbol{\nu}} - v\cdot \nabla_x, \quad \text{with    }\, 
{\bf G_1}\DT=(G_{1,1}\DT, \cdots, G_{1,i}\DT, \cdots, G_{1,N}\DT) \end{equation} is hypodissipative. Notice that 
\begin{align}\label{splitting.hypo.reg}
{\bf G} = {\bf A}\DT + {\bf G_1}\DT. 
\end{align}
The notation ${\bf \Pi_G}$ is the orthogonal projection onto $\text{Ker}({\bf G})$ in $L^2_{x,v}({\bf M}^{-1/2})$. 

Recall the shorthand notation $$\langle v \rangle = \sqrt{1+|v|^2}\,, $$
and the function spaces that we will use: 
\begin{align}\label{function.spaces} & \|{\bff}\|_{L_{x, v}^{\infty}(\mathbf{W})}=\sum_{i=1}^{N}\left\|f_{i}\right\|_{L_{x, v}^{\infty}\left(W_{i}\right)},  
\quad \left\|f_{i}\right\|_{L_{x, v}^{\infty}\left(W_{i}\right)}=\sup _{(x, v) \in \mathbb{T}^{3} \times \mathbb{R}^{3}}\left(\left|f_{i}(x, v)\right| W_{i}(v)\right), \\[4pt]
& \|{\bff}\|_{L_{v}^{1} L_{x}^{\infty}(\mathbf{W})}=\sum_{i=1}^{N}\left\|f_{i}\right\|_{L_{v}^{1} L_{x}^{\infty}\left(W_{i}\right)}, \quad
\left\|f_{i}\right\|_{L_{v}^{1} L_{x}^{\infty}\left(W_{i}\right)}=\left\|\sup _{x \in \mathbb{T}^{3}}\left|f_{i}(x, v)\right| W_{i}(v)\right\|_{L_{v}^{1}}, \notag
\end{align}
where $\mathbf{W}=\left(W_{1}, \ldots, W_{N}\right) : \mathbb{R}^{3} \rightarrow \mathbb{R}^{+}$ is a strictly positive measurable function in $v$. 
\hspace{1cm}

Denote $\pn f := \partial_z^n f$. 
The following theorem, which is our main result of Section \ref{sec:3}, gives the existence, Sobolev regularity and long-time behavior of the solution in the random space. 
\begin{theorem} 
\label{Prop-Main} 
Under the assumptions (H1)--(H5),
$\exists\, \eta_k$, $C_k$ and $\lambda_k >0$ such that for any $\pn {\bf f_0}\in L_v^1 L_x^{\infty}(\sp)$ satisfying ${\bf \Pi_G}(\pn {\bf f_0})=0$
{\color{black}for all $z$, that is, for $0\leq n \leq r$, 
$$||\pn {\bf f_0}||_{L_v^1 L_x^{\infty}(\sp)}\leq \eta_k, $$}
then there exists $\pn {\bff} \in L_v^1 L_x^{\infty}(\sp)$ satisfying ${\bf \Pi_G}(\pn \bff)=0$ for all $z$, which is a solution to the sensitivity system
\begin{equation}\label{f-der}
\partial_t (\pn f_i) = \pn G_i({\bff}) + \pn Q_i({\bff}), \qquad \pn {\bff}(t=0) = \pn {\bf f_0}\,, \end{equation}
such that for all $z$,
$$ ||\pn {\bff}||_{L_v^1 L_x^{\infty}(\sp)} \leq C_k\, e^{-\lambda_k t}. $$  
As a consequence, ${\pn \bff}$ satisfies for all $z$, 
$$ ||\pn {\bff}||_{L_v^1 L_x^{\infty}(\sp)L_z^{\infty}} \leq C_k\, e^{-\lambda_k t}, $$ 
where the constant $C_k$ depends on the initial data of $\partial^l {\bf f_0}$ for $l=0, \cdots, n$. 
\end{theorem}
\hspace{1cm}

Since we need the following Lemmas given in \cite{BD16} in the proof for the main Theorem \ref{Prop-Main}, we paraphrase them below. 
For each fixed  $z\in I_z$, Lemmas \ref{L0}, \ref{L1} and Lemma \ref{L2} are the same as Lemma 6.2, 6.3 and 6.6 of \cite{BD16}, respectively. 
\begin{lemma}
 \label{L0}
 For any $k$ in $\mathbb N$, $\beta>0$ and $\delta \in (0,1)$, $\exists\, C_A >0$ such that for all ${\bff}$ in 
 $L_v^1 L_x^{\infty}(\sp)$, 
 $$ ||{\bf A}\DT({\bff})||_{L_{x,v}^{\infty}(\langle v \rangle^{\beta}\boldsymbol{M}^{-1/2})} \leq C_A\, ||{\bff}||_{L_v^1 L_x^{\infty}(\sp)}. $$
 \end{lemma}

\begin{lemma}
 \label{L1}
There exists $k_0 \in \mathbb N$ such that for $k \geq k_0$, one can choose $\delta_k >0$ such that 
$0<C_B(k, \delta_k)<1$ and for all ${\bf f} \in L_v^1 L_x^{\infty}(\spnu)$, 
\begin{equation}\label{C-B} ||{\bf B}\DT({\bf f})||_{L_v^1L_x^{\infty}(\sp)} \leq C_B\, ||{\bf f}||_{L_v^1L_x^{\infty}(\spnu)}. 
\end{equation}
\end{lemma}

\begin{lemma}
\label{L2}
Define $\widetilde {\bf Q}(\bf f, \bf g)$ by 
$$ \forall 1\leq i \leq N, \qquad \widetilde Q_i({\bf f}, {\bf g}) = \frac{1}{2} \sum_{j=1}^N \left( Q_{ij}(f_i, g_j) + Q_{ij}(g_i, f_j)\right). $$
\textcolor{black}{Then} for all ${\bf f}, {\bf g}$ such that $ \widetilde Q_i({\bf f}, {\bf g})$ \textcolor{black}{is well-defined, the latter} belongs to $\left[\text{Ker}({\bf L})\right]^{\perp}$, \textcolor{black}{and}
$\exists\, C_Q>0$ such that $\forall 1\leq i \leq N$ and each $\bf f$ and $\bf g$, 
\begin{align}
\label{C-Q}
\begin{split}
&\displaystyle ||\widetilde Q_i({\bf f}, {\bf g})||_{L_v^1L_x^{\infty}(\sp)} \leq C_Q \left[ ||f_i||_{L_v^1 L_x^{\infty}(\sp)} 
||{\bf g}||_{L_v^1 L_x^{\infty}(\spnu)} \right. \\[4pt]
&\displaystyle \qquad\qquad\qquad\qquad\qquad\qquad\left.  +
||f_i||_{L_v^1L_x^{\infty}(\nu_i \langle v\rangle^k)} ||{\bf g}||_{L_v^1 L_x^{\infty}(\sp)} \right]. 
\end{split}
\end{align}
\end{lemma}
\textcolor{black}{
The strategy of the proof is to introduce a new adaptation of the factorization method of Gualdani, Mischler and Mouhot \cite{GMM} to our probabilistic setting studied in this paper. 
The core idea is to decompose the full linear operator $\bf G$ (defined in \eqref{splitting.G}) into the hypodissipative operator ${\bf A}\DT$ (see \eqref{splitting.hypo.reg}) and the regularizing operator ${\bf G_1}\DT$ (see \eqref{splitting.hypo.reg}), and to decompose the sensitivity system \eqref{f-der} into a system of equations, such that the hypodissipative and regularizing effects of the operators can be used to obtain the result of Theorem \ref{Prop-Main}.}

\textcolor{black}{The additional challenge here in our framework with uncertainty compared to the deterministic results in \cite{Marc, BD16, GMM} is to find a way of handling the extra high-order derivatives in the random parameter $z$, which naturally appear from the fact that we introduce uncertainty into the model. Thus, the main difference and new challenge in our work compared to all the previous works on the deterministic problem is that a new decomposition, denoted by ${\bf g} = {\bf g_1}+ {\bf g_2}$, for each order $z$-derivative of the distribution function has to be introduced.} \textcolor{black}{One needs to carefully design this new decomposition into the coupled system for ${\bf g_1}$, ${\bf g_2}$ (see equations \eqref{G1}--\eqref{G2}) such that the hypodissipative and regularising properties for the new operators (see the definitions for $A_{b^k}^{(\delta)}$ and  $B_{b^k}^{(\delta)}$ in equation \eqref{New-op}) can be proved and used in a similar way as in the deterministic problems. 
Finally, a suitable induction in the order of $z$-derivatives needs to be applied. }

\textcolor{black}{Compared to the previous work on the sensitivity analysis for a class of (single-species) collisional kinetic equations with multiple scales and random inputs \cite{LJ18}, we want to highlight the following differences in this work: 
First, here we conduct the sensitivity analysis for the multi-species Boltzmann system, while \cite{LJ18} studied a class of single-species kinetic equations, 
including the Boltzmann equation with random initial data and collision kernel. 
Second, here we rigorously prove the existence of solutions to the sensitivity equations, and its exponential decay to the equilibrium in the norm $||\cdot||_{L_v^1 L_x^{\infty}(\sp)L_z^{\infty}}$. }

\subsection{The proof of Theorem \ref{Prop-Main}}
\label{subsec:2}

We shall prove Theorem \ref{Prop-Main} by induction. The deterministic case of $n=0$ is shown in \cite{BD16}. 
Now assume that Proposition \ref{Prop-Main} holds for all $0\leq m\leq n-1$ with $n\geq 1$, we shall prove that the result holds for $m=n$. 

First, one needs to calculate $\partial^n G_i({\bff})$ and $\partial^n Q_i({\bff})$. 
Denote 
\begin{align}
\label{New-op} 
\begin{split}
&\displaystyle A_{b^k, i}\DT(\pl{\bff}) = \sum_{j=1}^N \int_{\mathbb R^3\times\mathbb S^2} \Theta_{\delta}
(M_j^{\prime\ast}\pl f_i^{\prime} + M_i^{\prime}\pl f_j^{\prime\ast} - M_i \pl f_j^{\ast})\, C_{ij}^{\Phi}\, |v-v^{\ast}|^{\gamma}\, \partial^k b_{ij}(\cos\theta, z)\, d\sigma dv^{\ast}, \\[4pt]
&\displaystyle B_{b^k, i}\DT(\pl{\bff})  = \sum_{j=1}^N \int_{\mathbb R^3\times\mathbb S^2}  
(1 - \Theta_{\delta}) (M_j^{\prime\ast}\pl f_i^{\prime} + M_i^{\prime}\pl  f_j^{\prime\ast} - M_i \pl f_j^{\ast})\,C_{ij}^{\Phi}\, |v-v^{\ast}|^{\gamma}\,\partial^k b_{ij}(\cos\theta, z)\, d\sigma dv^{\ast}. 
\end{split}
\end{align}
Compared with $A\DT$, $B\DT$ shown in \eqref{A-B}, the only difference in $A_{b^k}\DT$, $B_{b^k}\DT$ is that one replaces the angular part of the kernel to be $\partial^k b_{ij}$ here instead of $b_{ij}$.
The $n$-order $z$-derivative of the ${\bf G}$ operator is given by 
\begin{align}
\label{G-n}
\begin{split}
\displaystyle \pn G_i({\bff}) &= \pn A_i\DT({\bff}) + \pn B_i\DT({\bff})  - \pn (\nu_i f_i) - v\cdot\nabla_x (\pn f_i) \\[4pt]
\displaystyle&= A_i\DT(\pn {\bff}) + B_i\DT(\pn {\bff}) - \nu_i \, \pn f_i - v\cdot\nabla_x (\pn f_i)  \\[4pt]
\displaystyle& \quad +\sum_{k=1}^n \nk \left[ A_{b^k, i}\DT(\partial^{n-k}{\bff}) + B_{b^k,i}\DT(\partial^{n-k}{\bff}) - \partial^k \nu_i \,\partial^{n-k}f_i \right]  \\[4pt]
\displaystyle& = A_i\DT(\pn{\bff}) + G_{1,i}\DT(\pn{\bff}) + \sum_{k=1}^n \nk \left[ A_{b^k, i}\DT(\partial^{n-k}{\bff}) + B_{b^k,i}\DT(\partial^{n-k}{\bff}) - \partial^k \nu_i\, \partial^{n-k}f_i \right]{\textcolor{black}.}
\end{split}
\end{align}
Denote $$Q_{ij}^{b^k}(f_i, f_j) = \int_{\mathbb R^3\times \mathbb S^2} C_{ij}^{\Phi}\, |v-v^{\ast}|^{\gamma}\,\partial^k b_{ij}(\cos\theta, z)\, (f_i^{\prime}f_j^{\prime\ast} - 
f_i f_j^{\ast}) d\sigma dv^{\ast}. $$
Then the $n$-order $z$-derivative of the collision operator $Q_{ij}$ is
\begin{align*}
\displaystyle \pn Q_{ij}(f_i, f_j)& = \sum_{l=0}^n \nl \int_{\mathbb R^3\times \mathbb S^2}  \partial^{n-l} B_{ij}
\sum_{m=0}^l \lm \left(\partial^m f_i^{\prime}\,\partial^{l-m}f_j^{\prime\ast} - \partial^m f_i \,\partial^{l-m}f_j^{\ast}\right) d\sigma dv^{\ast} \\[4pt]
\displaystyle &= \sum_{l=0}^n \sum_{m=0}^l \nl \lm Q_{ij}^{b^{n-l}}(\partial^m f_i, \partial^{l-m}f_j) \\[4pt]
\displaystyle & = \sum_{l=0}^{n-1}\sum_{m=0}^l \nl \lm Q_{ij}^{b^{n-l}}(\partial^m f_i, \partial^{l-m}f_j)  + 
\sum_{m=1}^{n-1}\nm Q_{ij}(\partial^m f_i, \partial^{n-m}f_j)  \\[4pt]
\displaystyle & \quad + Q_{ij}(f_i, \pn f_j) + Q_{ij}(\pn f_i, f_j), 
\end{align*}
thus 
\begin{align}
\label{Q-n}
\begin{split}
&\displaystyle\quad \pn Q_i (f_i, f_j) = \sum_{j=1}^N \pn Q_{ij}(f_i, f_j) \\[4pt]
&\displaystyle= \underbrace{\sum_{j=1}^N \sum_{l=0}^{n-1}\sum_{m=0}^l \nl \lm Q_{ij}^{b^{n-l}}(\partial^m f_i, \partial^{l-m}f_j) 
+ \sum_{j=1}^N \sum_{m=1}^{n-1}\nm Q_{ij}(\partial^m f_i, \partial^{n-m}f_j)}_{\text{Term}\, \ostar} + 2 \widetilde Q_i(\pn{\bff}, {\bff}). 
\end{split}
\end{align}

Combine \eqref{f-der}, \eqref{G-n} and \eqref{Q-n}, then ${\bf g} := \pn \bff$ satisfies for each $z$ the equation 
\begin{align}\label{G-eqn} 
\begin{split}
\partial_t g_i  & =  
{\textcolor{black}G_i({\bf g})}
+ \sum_{k=1}^n \nk \left[ A_{b^k, i}\DT(\partial^{n-k}{\bff}) + B_{b^k,i}\DT(\partial^{n-k}{\bff}) - \partial^k \nu_i\, \partial^{n-k}f_i \right] \\[6pt] & \quad + 2 \widetilde Q_i({\bfg}, {\bff}) + \text{Term}\, \ostar, 
\qquad {\bf g}(0,x,v,z) = {\bf g_0}(x,v,z) = \pn {\bf f_0}(x,v,z). 
\end{split}
\end{align}

\hspace{1cm}

\noindent {\bf Decomposition: }
In the form of ${\bf g} = {\bf g_1} + {\bf g_2}$ with ${\bf g_1} \in L_v^1 L_x^{\infty}(\sp)$ and 
${\bf g_2} \in L_{x,v}^{\infty}(\langle v \rangle^{\beta}{\bf \mu}^{-1/2})$, then $({\bf g_1}, {\bf g_2})$ satisfy the following system of equations
\begin{align}
\label{G1}\partial_t g_{1,i} & = G_{1,i}^{(\delta)}({\bf g_1})  + \sum_{k=1}^n \nk \left[ B_{b^k,i}\DT(\partial^{n-k}{\bff}) - \partial^k \nu_i\, \partial^{n-k}f_i \right] \\[6pt]
&\quad + 2 \widetilde Q_i({\bf g_1} + {\bf g_2}, {\bff}) + \text{Term}\, \ostar,  \qquad {\bf g_1}(0,x,v,z) = {\bf g_0}(x,v,z), \notag \\[8pt]
\label{G2}\partial_t  g_{2,i} &= G_i({\bf g_2}) + A_i^{(\delta)}({\bf g_1}) + \sum_{k=1}^n \nk A_{b^k, i}\DT(\partial^{n-k}{\bff}), 
\qquad {\bf g_2}(0,x,v,z)=0. 
\end{align}
The above decomposition of the solution ${\bf g}={\bf g_1}+{\bf g_2}$ follows \cite{BD16}, which also adopted the idea in \cite{GMM} for the single-species Boltzmann equation. Compared to the deterministic case studied in \cite{BD16}, the differences here are the last three terms on the right-hand-side of \eqref{G1}, 
which appear due to the uncertainty dependence, 
 and the last term on the right-hand-side of \eqref{G2}. They need to be grouped properly in the equation for ${\bf g_1}$ or ${\bf g_2}$. 

\hspace{2cm}

First, we show a simple Lemma: 
\begin{lemma}
\label{Q-I}
Denote 
$$\chi_n = 
\left\{\begin{array}{ll} 1, & n \, \text{ is even} \\ 
0, & n \, \text{ is odd}
\end{array}\right. 
$$
One can write 
\begin{equation}\label{Q-eqn} \sum_{j=1}^N\sum_{m=1}^{n-1}\nm Q_{ij}(\partial^m f_i, \partial^{n-m}f_j) = 
2 \sum_{k=1}^{\lfloor\frac{n-1}{2}\rfloor}\nm \widetilde Q_i (\partial^m {\bff}, \partial^{n-m}{\bff}) 
+ \chi_n \binom{n}{\frac n 2}\widetilde Q_i (\partial^{\frac n 2}{\bff},\partial^{\frac n 2}{\bff}). \end{equation}
Also, we have the estimate for $0\leq \ell \leq n$, 
\begin{equation}\label{Q-eqn2} \sum_{j=1}^N  || Q_{ij}^{b^{\ell}}(f_i,  g_j)||_{L_v^1 L_x^{\infty}(\sp)}
\leq \widetilde C_Q \left[ ||f_i||_{L_v^1 L_x^{\infty}(\sp)} ||{\bfg}||_{L_v^1 L_x^{\infty}(\sp \boldsymbol{\nu})} + ||f_i||_{L_v^1 L_x^{\infty}(\nu_i\sp)} ||{\bf g}||_{L_v^1 L_x^{\infty}(\sp)}\right].  \end{equation}
\end{lemma}
The proof is given in the Appendix. By Lemma \ref{L2}, \eqref{Q-eqn} implies that 
\begin{align}  
&\quad\left|\left| \sum_{j=1}^N \sum_{m=1}^{n-1} \nm Q_{ij}(\partial^m f_i , \partial^{n-m}f_j)\right|\right|_{L_v^1 L_x^{\infty}(\sp)}\notag \\[2pt]
& \leq 2 C_Q \sum_{k=1}^{\lfloor\frac{n-1}{2}\rfloor}\nm \left[ ||\partial^{m}f_i||_{L_v^1 L_x^{\infty}(\sp)} ||\partial^{n-m}{\bff}||_{L_v^1 L_x^{\infty}(\spnu)}
 + ||\partial^{m}f_i||_{L_v^1 L_x^{\infty}(\nu_i\sp)} ||\partial^{n-m}{\bff}||_{L_v^1 L_x^{\infty}(\sp)}\right] \notag\\[2pt]
&\label{Q-ij}\quad + \chi_n \binom{n}{\frac n 2} C_Q  \left[ ||\partial^{\frac n 2}f_i||_{L_v^1 L_x^{\infty}(\sp)} ||\partial^{\frac n 2}{\bff}||_{L_v^1 L_x^{\infty}(\spnu)}+ ||\partial^{\frac n 2}f_i||_{L_v^1 L_x^{\infty}(\nu_i\sp)} ||\partial^{\frac n 2}{\bff}||_{L_v^1 L_x^{\infty}(\sp)}\right]. 
\end{align}

In ``Term $\ostar$", the second term is exactly the left-hand-side of \eqref{Q-eqn}. 
By using the assumption \eqref{b-Assump} and Lemma \ref{Q-I}, the first term is estimated by
{\footnotesize\begin{align*}
& \quad\left|\left| \sum_{j=1}^N \sum_{l=0}^{n-1}\sum_{m=0}^l \nl \lm Q_{ij}^{b^{n-l}}(\partial^m f_i, \partial^{l-m}f_j) \right|\right|_{L_v^1 L_x^{\infty}(\sp)} \\[4pt]
& \leq \sum_{l=0}^{n-1} \nl \left\{\sum_{j=1}^N \sum_{m=1}^{l-1} \lm  \left|\left| Q_{ij}^{b^{n-l}}(\partial^m f_i, \partial^{l-m}f_j) \right|\right|_{L_v^1 L_x^{\infty}(\sp)}
+ \left|\left|\sum_{j=1}^N \left(Q_{ij}(f_i, \partial^l f_j) + Q_{ij}(\partial^l f_i, f_j)\right)\right|\right|_{L_v^1 L_x^{\infty}(\sp)} \right\}  \\[4pt]
& \leq 
\sum_{l=0}^{n-1} \nl \left\{\sum_{m=1}^{l-1}\lm \widetilde C_Q\left[ ||\partial^m f_i ||_{L_v^1 L_x^{\infty}(\sp)} ||\partial^{l-m}{\bff}||_{L_v^1 L_x^{\infty}(\sp\boldsymbol{\nu})} +  ||\partial^m f_i ||_{L_v^1 L_x^{\infty}(\sp\nu_i)} ||\partial^{l-m}{\bff}||_{L_v^1 L_x^{\infty}(\sp)}\right] \right.  \\[4pt]
&\left. \qquad\qquad\quad + 2 || \widetilde Q_i({\bff}, \partial^l{\bff})||_{L_v^1 L_x^{\infty}(\sp)} \right\}. 
\end{align*}}
Thus ``Term $\ostar$" can be bounded by 
{\small\begin{align}
\label{Star}
\begin{split}
&\quad ||\text{Term}\, \ostar||_{L_v^1 L_x^{\infty}(\sp)} \\[4pt]
&\leq \sum_{l=0}^{n-1} \nl \left\{ \sum_{m=1}^{l-1}\lm \widetilde C_Q\left[ ||\partial^m f_i ||_{L_v^1 L_x^{\infty}(\sp)} ||\partial^{l-m}{\bff}||_{L_v^1 L_x^{\infty}(\sp\boldsymbol{\nu})} +  ||\partial^m f_i ||_{L_v^1 L_x^{\infty}(\sp\nu_i)}||\partial^{l-m}{\bff}||_{L_v^1 L_x^{\infty}(\sp)}\right]\right.  \\[4pt]
&\left. \qquad\qquad\quad + 2 ||\widetilde Q_i({\bff}, \partial^l{\bff})||_{L_v^1 L_x^{\infty}(\sp)} \right\} \\[4pt]
&\quad + 2 \sum_{m=1}^{\lfloor\frac{n-1}{2}\rfloor}\nm ||\widetilde Q_i (\partial^m {\bff}, \partial^{n-m}{\bff})||_{L_v^1 L_x^{\infty}(\sp)}
+ \chi_n\binom{n}{\frac n 2}||\widetilde Q_i (\partial^{\frac n 2}{\bff},\partial^{\frac n 2}{\bff})||_{L_v^1 L_x^{\infty}(\sp)}.  
\end{split}
\end{align}}
 
 Another thing we would like to mention before starting the main steps of the proof: 
\cite[Proposition 6.1 and at the end of section 6.1.2]{BD16} shows that the solution ${\bf f}$ is small in the following sense, and one can assume that
\begin{equation}\label{F}\int_0^t ||{\bf f}||_{L_v^1L_x^{\infty}(\spnu)}ds \leq \tau_1, \qquad
 ||{\bf f}||_{L_t^{\infty}L_v^1L_x^{\infty}(\sp)} \leq \tau_2, \end{equation} 
 where $\tau_1$, $\tau_2$ are constants depending on the initial data $||{\bf f_0}||_{L_v^1L_x^{\infty}(\sp)}$ and an exponential decay factor 
$e^{-\lambda_k t}$. 

\subsubsection{Step 1: discussion for ${\bf g_1}$}
In a similar spirit as \cite[Proposition 6.7]{BD16}, we will show that 
\begin{proposition}
\label{PG-A}
\textcolor{black}{Let $k>k_0$,} and for all $z$, let ${\bf g_0}\in L_v^1L_x^{\infty}(\sp)$ satisfy ${\bf \Pi_{G}(g_0)}=0$ and 
${\bf h} = {\bf h}(t,x,v,z) \in L_t^{\infty}L_v^1L_x^{\infty}(\spnu)$. Moreover, let $\tau_1$, $\tau_2$ in \eqref{F} be small enough such that
\begin{equation}\label{Assump}\max \{ 4 C_Q \tau_1, 2(C_B + 2 C_Q \tau_2) \} <1. 
\end{equation}
Then there exists $\eta_1$, $\lambda_1$ such that 
for all $z$, if 
$$ ||{\bf g_0}||_{L_v^1L_x^{\infty}(\sp)} \leq \eta_1, \quad \text{and   } \exists C,\lambda \quad \text{such that   }\quad   
||{\bf h}(t, z)||_{L_v^1L_x^{\infty}(\spnu)} \leq C\, ||{\bf g_0}||_{L_v^1 L_x^{\infty}(\sp)} \, e^{-\lambda t}, $$
then there exists a function ${\bf g_1}$ in $L_t^{\infty}L_v^1L_x^{\infty}(\sp)$ such that for $1\leq i \leq N$, 
\begin{align}
\begin{split}\label{G1-1} 
\partial_t g_{1,i} & = G_{1,i}^{(\delta)}({\bf g_1})  + \sum_{\ell=1}^n \binom{n}{\ell} \left[ B_{b^{\ell},i}\DT(\partial^{n-\ell}{\bff}) - \partial^{\ell} \nu_i\, \partial^{n-\ell}f_i \right] \\[6pt]
&\quad + 2 \widetilde Q_i({\bf g_1} + {\bf h}, {\bff}) + \text{Term}\, \ostar,  \qquad {\bf g_1}(0,x,v,z) = {\bf g_0}(x,v,z). \end{split}
\end{align}
In addition, for all $z$, solution ${\bf g_1}$ satisfies for all \textcolor{black}{$t\geq 0$}
$$ ||{\bf g_1}(t, z)||_{L_v^1L_x^{\infty}(\sp)} \leq C_1\, e^{-\lambda_1 t}. $$
The constants $C_1$, $\eta_1$ and $\lambda_1$ depend on $n$, $k$ and the collision kernel. $C_1$ also depends on the initial data ${\bf g_0}$ and $\partial^k {\bf f_0}$ for $1 \leq k \leq n$. The constants 
$C_B$, $C_Q$, $\tau_1$, $\tau_2$ are shown in \eqref{C-B}, \eqref{C-Q} and \eqref{F}, respectively. 
\end{proposition}

\begin{proof}
{\bf Step (i): a priori exponential decay.}\, This part follows the main structure of \cite[Proof of Proposition 6.7, page 1430]{BD16} thus we omit some details. Using that the transport \textcolor{black}{part} gives null contribution and multiplicative part gives a negative contribution, similar to \cite[inequality (6.13)]{BD16}, one gets
\begin{align*}
& \quad\frac{d}{dt}||g_{1,i}||_{L_v^1 L_x^{\infty}(\sp)}  \leq - ||g_{1,i}||_{L_v^1 L_x^{\infty}(\sp\nu_i)} + 
\sum_{\ell=1}^n \binom{n}{\ell} ||\partial^{n-\ell}f_i ||_{L_v^1L_x^{\infty}(\sp \nu_i)}
+ ||{\bf B}_i^{(\delta)}({\bf g}_{1})||_{L_v^1 L_x^{\infty}(\sp)}  \\[2pt]
& + \sum_{\ell=1}^n \binom{n}{\ell} ||B_{b^{\ell},i}\DT(\partial^{n-\ell}{\bff})||_{L_v^1L_x^{\infty}(\sp)} + 2 ||\widetilde Q_i({\bf g_1} + {\bf h}, {\bff})||_{L_v^1 L_x^{\infty}(\sp)} + ||\text{Term}\, \ostar ||_{L_v^1 L_x^{\infty}(\sp)}. 
\end{align*}
In analogy to Lemma \ref{L1}, due to our assumption $|\partial_z^{\ell}b_{ij}| \leq C_b$ in \eqref{b-Assump}, then
\begin{equation}\label{New-B} ||B_{b^{\ell},i}\DT(\partial^{n-\ell}{\bff})||_{L_v^1L_x^{\infty}(\sp)} \leq C_B\, ||\partial^{n-\ell}{\bff}||_{L_v^1 L_x^{\infty}(\sp \boldsymbol{\nu})}, \end{equation}
where the same $C_B$ is generated as in the deterministic case satisfying $0< C_B < 1$. 
This is due to our assumptions \eqref{bb} and \eqref{b-Assump} that $b_{ij}$ and 
$|\partial_z^k b_{ij}|$ share the same upper bound $C_b$, which determines $C_B$ in Lemma \ref{C-B}. 
We use Lemma \ref{L1} to control ${\bf B}\DT$, \eqref{New-B} to control ${\bf B}\DT_{b^{\ell}}$ and Lemma \ref{L2} to control $\widetilde Q_i$. 
Using estimate \eqref{Q-ij}, one has
\begin{align*}  
&\quad\frac{d}{dt}||{\bf g_1}||_{L_v^1 L_x^{\infty}(\sp)} \\[4pt]
&\leq - \left[  1 - C_B - 2 C_Q  ||{\bf f}||_{L_v^1 L_x^{\infty}(\sp)}\right] ||{\bf g_1}||_{L_v^1 L_x^{\infty}(\spnu)} 
+ \sum_{\ell=1}^n \binom{n}{\ell} ||\partial^{n-\ell}f_i ||_{L_v^1L_x^{\infty}(\langle v \rangle^k \nu_i)} \\[4pt]
&\quad  + \sum_{\ell=1}^n \binom{n}{\ell} C_B\, ||\partial^{n-\ell}{\bff}||_{L_v^1 L_x^{\infty}(\langle v \rangle^k\boldsymbol{\nu})}
+ 2 C_Q \left[ ||{\bf g_1}||_{L_v^1 L_x^{\infty}(\sp)} ||{\bf f}||_{L_v^1 L_x^{\infty}(\spnu)} \right. \\[4pt]
&\qquad\quad + \left. ||{\bf h}||_{L_v^1L_x^{\infty}(\sp)} ||{\bf f}||_{L_v^1 L_x^{\infty}(\spnu)} + ||{\bf h}||_{L_v^1 L_x^{\infty}(\spnu)} ||{\bf f}||_{L_v^1 L_x^{\infty}(\sp)}\right] \\[4pt]
 &\quad + \text{``RHS of \eqref{Star}"}, 
\end{align*} 
where the last term ``RHS of \eqref{Star}" is bounded by products of lower-order (up to $(n-1)$-th) partial $z$-derivatives of ${\bff}$, according to Lemma \ref{L2}, \textcolor{black}{Lemma \ref{Q-I} and \eqref{Q-ij}}. 
Since $C_B<1$, $||{\bf g_0}||_{L_v^1 L_x^{\infty}(\spnu)}$ is sufficiently small, due to the exponential decay of 
$||{\bf h}(t)||_{L_v^1L_x^{\infty}(\spnu)}$, smallness of ${\bff}$ shown in \cite{BD16} and smallness conditions for all $||\partial^m {\bff}||_{L_v^1 L_x^{\infty}(\sp)}$ ($1\leq m \leq n-1$) assumed by induction, the Gr\"onwall's lemma yields the exponential decay of $||{\bf g_1}||_{L_v^1 L_x^{\infty}(\sp)}$. 
Note that $\nu_i$ is equivalent to $\langle v \rangle^{\gamma}$, thus 
$||\partial^m {\bff}||_{L_v^1 L_x^{\infty}(\sp \boldsymbol{\nu})}$ ($1\leq m \leq n-1$) is also small. 
\\[2pt]

\noindent {\bf Step (ii): existence.}\, Let ${\bf g_1}^{({\bf {0}})} = 0$ and consider the following iteration on equation \eqref{G1-1} with 
$p\in \mathbb N$: 
\begin{align} 
\label{g1-N}
\begin{split}
 \partial_t g_{1,i}^{(p+1)} + v\cdot\nabla_x g_{1,i}^{(p+1)} = & - \nu_i(v) (g_{1,i}^{(p+1)}) + 
B_i({\bf g_1}^{(p)}) + 2\widetilde Q_i({\bf g_1}^{(p)} + {\bf h}, {\bff})  \\[4pt]
& + \sum_{\ell=1}^n \binom{n}{\ell} \left[ B_{b^{\ell},i}(\partial^{n-\ell}{\bff}) - \partial^{\ell} \nu_i\, \partial^{n-\ell}f_i \right]
+ \text{Term } \ostar, 
\end{split}
\end{align}
with the initial data ${\bf g_1}^{(p+1)}(0,x,v,z) = {\bf g_0}$. We omit including the superscript $\delta$ in $ {\bf B}^{(\delta)}$ here. 
Note that in \eqref{g1-N}, the last two terms on the right-hand-side do not involve the time iteration index $p$ of the scheme. 
Our goal is to show that $\left({\bf g_1}^{(p)}\right)_{p\in \mathbb N}$ is a Cauchy sequence in $L_t^{\infty}L_v^1L_x^{\infty}(\sp)$. 

By the Duhamel formula along the characteristics for all $i$, 
\begin{align}\label{Duh1}
\begin{split} g_{1,i}^{(p+1)}(t,x,v,z) = & e^{-\nu_i(v) t} g_{0, i} + \int_0^t e^{-\nu_i(v)(t-s)} \Big\{ B_i({\bf g_1}^{(p)})  +  
2 \widetilde Q_i({\bf g_1}^{(p)}+ {\bf h}, {\bf f}) \\[2pt]
&\qquad\qquad\quad + \sum_{\ell=1}^n \binom{n}{\ell} \left[ B_{b^{\ell},i}(\partial^{n-\ell}{\bff}) - \partial^{\ell} \nu_i\, \partial^{n-\ell}f_i \right]
+ \text{Term} \ostar \Big\} (x-sv, v)\, ds, 
\end{split}
\end{align}
where $g_{0,i}(x,v,z)$ is the $i$-th component of the initial data ${\bf g_0}$. 
Similarly we write
\begin{align}\label{Duh2}
\begin{split} g_{1 i}^{(p)}(t,x,v,z) =  & e^{-\nu_i(v) t} g_{0, i} + \int_0^t e^{-\nu_i(v)(t-s)} \Big\{ B_i({\bf g_1}^{(p-1)}) 
 + 2 \widetilde Q_i({\bf g_1}^{(p-1)}+ {\bf h}, {\bf f}) \\[2pt]
 &\qquad\qquad\quad + \sum_{\ell=1}^n \binom{n}{\ell} \left[ B_{b^{\ell},i}(\partial^{n-\ell}{\bff}) - \partial^{\ell} \nu_i\, \partial^{n-\ell}f_i \right]
+ \text{Term} \ostar \Big\} (x-sv, v)\, ds. 
 \end{split}
\end{align}
Since we are in the case of hard potentials and Maxwellian molecules, we know that $\nu_i(v)\geq \nu_0\color{black}{>0}$. 
Subtract (\ref{Duh2}) from (\ref{Duh1}), take the $L_v^1L_x^{\infty}(\sp)$-norm of $({\bf g_1}^{(p+1)} - {\bf g_1}^{(p)})$ and 
sum over $i$, by using the relation $$ \widetilde{\bf Q}({\bf g_1}^{(p)}+ {\bf h}, {\bf f}) - \widetilde {\bf Q}({\bf g_1}^{(p-1)}+ {\bf h}, {\bf f})= 
\widetilde{\bf Q}({\bf g_1}^{(p)} - {\bf g_1}^{(p-1)}, {\bf f}), $$ 
one gets for each $z$, 
\begin{align}
\label{GE1}
\begin{split}
&\displaystyle \left|\left|{\bf g_1}^{(p+1)}(t) - {\bf g_1}^{(p)}(t)\right|\right|_{L_v^1L_x^{\infty}(\sp)} \\[4pt]
\displaystyle \leq &\int_0^t e^{-\nu_0 (t-s)}\, \left|\left|{\bf B}({\bf g_1}^{(p)} - {\bf g_1}^{(p-1)}) 
+ 2 \widetilde {\bf Q}({\bf g_1}^{(p)} - {\bf g_1}^{(p-1)}, {\bf f})\right|\right|_{L_v^1L_x^{\infty}(\sp)} \, ds \\[4pt]
\displaystyle \leq &\left[ C_B + 2 C_Q ||{\bf f}||_{L_t^{\infty}L_v^1L_x^{\infty}(\sp)} \right] 
\int_0^t e^{-\nu_0(t-s)} \left|\left| {\bf g_1}^{(p)}(s) - {\bf g_1}^{(p-1)}(s) \right|\right|_{L_v^1L_x^{\infty}(\spnu)} ds \\[4pt]
&\displaystyle + 2 C_Q \int_0^t ||{\bf f}||_{L_v^1L_x^{\infty}(\spnu)} ds \cdot 
\sup_{s\in [0,t]} \left|\left|{\bf g_1}^{(p)}(s) - {\bf g_1}^{(p-1)}(s)\right|\right|_{L_v^1L_x^{\infty}(\sp)}, 
\end{split}
\end{align}
where Lemma \ref{L1} and Lemma \ref{L2} on estimates of the operator ${\bf B}$ and $\widetilde{\bf Q}$ are used. 

On the other hand, 
\begin{align}
&\displaystyle \int_0^t \left|\left|{\bf g_1}^{(p+1)}(s) - {\bf g_1}^{(p)}(s)\right|\right|_{L_v^1L_x^{\infty}(\spnu)} ds  \notag\\[4pt]
\displaystyle \leq &\sum_{i}\int_0^t \int_0^s \int_{\mathbb R^3} e^{-\nu_i(v)(s-s_1)}\nu_i(v) \sp  \notag\\[4pt]
\displaystyle &\qquad\qquad \cdot \left|\left|{\bf B}({\bf g_1}^{(p)} - {\bf g_1}^{(p-1)}) 
+ 2 \widetilde {\bf Q}({\bf g_1}^{(p)} - {\bf g_1}^{(p-1)}, {\bf f})\right|\right|_{L_x^{\infty}}(s_1) ds_1 ds  \notag\\[4pt]
\displaystyle = &\sum_{i}\int_0^t \int_{\mathbb R^3} \left(\int_{s_1}^t e^{-\nu_i(v)(s-s_1)}\nu_i(v) ds\right) \sp  \notag\\[4pt]
\displaystyle &\qquad\qquad \cdot \left|\left|{\bf B}({\bf g_1}^{(p)} - {\bf g_1}^{(p-1)}) 
+ 2 \widetilde {\bf Q}({\bf g_1}^{(p)} - {\bf g_1}^{(p-1)}, {\bf f})\right|\right|_{L_x^{\infty}}(s_1) ds_1 \notag\\[4pt]
\displaystyle \leq & \left[ C_B + 2 C_Q\, ||{\bf f}||_{L_t^{\infty}L_v^1L_x^{\infty}(\sp)} \right]  
\int_0^t \left|\left| {\bf g_1}^{(p)}(s_1) - {\bf g_1}^{(p-1)}(s_1) \right|\right|_{L_v^1L_x^{\infty}(\spnu)} ds_1 \notag\\[4pt]
&\label{GE2}\displaystyle + 2 C_Q \int_0^t ||{\bf f}||_{L_v^1L_x^{\infty}(\spnu)}ds_1 \cdot
\sup_{s\in [0,t]} \left|\left|{\bf g_1}^{(p)}(s) - {\bf g_1}^{(p-1)}(s)\right|\right|_{L_v^1L_x^{\infty}(\sp)}, 
\end{align}
where we used the fact that the integral in $s$ is bounded by 1; exchanged the integration domains in $s$ and $s_1$, 
and used Lemma \ref{L1} and Lemma \ref{L2} again. 

Adding up (\ref{GE1}) and (\ref{GE2}), by using (\ref{F}), one has
 \begin{align*}
 &\displaystyle  \left|\left|{\bf g_1}^{(p+1)}(t) - {\bf g_1}^{(p)}(t)\right|\right|_{L_v^1L_x^{\infty}(\sp)}  
 + \int_0^t \left|\left|{\bf g_1}^{(p+1)}(s) - {\bf g_1}^{(p)}(s)\right|\right|_{L_v^1L_x^{\infty}(\spnu)} ds \\[4pt]
 \displaystyle \leq &  4 C_Q \tau_1 \cdot \sup_{s\in [0,t]} \left|\left|{\bf g_1}^{(p)}(s) - {\bf g_1}^{(p-1)}(s)\right|\right|_{L_v^1L_x^{\infty}(\sp)} \\[4pt]
 &\displaystyle + 2 (C_B + 2 C_Q \tau_2) \int_0^t \left|\left| {\bf g_1}^{(p)}(s) - {\bf g_1}^{(p-1)}(s) \right|\right|_{L_v^1L_x^{\infty}(\spnu)} ds. 
 \end{align*}
Assumption \eqref{Assump} indicates that $\left({\bf g_1}^{(p)}\right)_{p\in \mathbb N}$ is a Cauchy sequence 
in $L_t^{\infty}L_v^1 L_x^{\infty}(\sp)$. Thus $\left({\bf g_1}^{(p)}\right)_{p\in \mathbb N}$ converges to a function 
${\bf g_1}$ in $L_t^{\infty}L_v^1 L_x^{\infty}(\sp)$. 
 \end{proof}
 
\subsubsection{Step 2: discussion for ${\bf g_2}$}
 As for ${\bf g_2}$, it satisfies the linear equation (\ref{G2}), which is in a similar form as \cite[equation (6.3)]{BD16} except for the last term involving lower order $z$-derivatives of ${\bff}$. We thereby mimic \cite[Proposition 6.8]{BD16} and get the following: 
\begin{proposition}
\label{PG-B}
Let ${\bf h} = {\bf h}(t,x,v,z)$ be in $L_t^{\infty}L_v^1L_x^{\infty}(\sp)$, if 
${\bf \Pi_{G}(g_2 + h) = 0}$ and for all $z$, 
$||{\bf h}(t, z)||_{L_v^1L_x^{\infty}(\sp)} \leq \eta_h\, e^{-\lambda_h t}$, 
then there exists a unique function ${\bf g_2} \in L_t^{\infty}L_{x,v}^{\infty}
(\langle v \rangle^{\beta}{\bf M}^{-1/2})$ to 
\begin{equation}\label{G2-2} \partial_t  g_{2,i} = G_i({\bf g_2}) + A_i\DT({\bf h}) + \sum_{k=1}^n \nk A_{b^k, i}\DT(\partial^{n-k}{\bff}), \qquad {\bf g_2}(0,x,v,z)=0. 
\end{equation}
Moreover, $\exists$ some constants $C_2>0$, $\lambda_2>0$ such that for all $z$, 
$$ ||{\bf g_2}(t, z)||_{L_{x,v}^{\infty}(\langle v \rangle^{\beta}{\bf M}^{-1/2})} \leq C_2 \, \eta_h\, e^{-\lambda_2 t}, $$
where $C_2$ depends on the initial data of $\partial^k {\bf f_0}$ for $1\leq k \leq n$. 
\end{proposition}
\noindent
The proof is similar to \cite{BD16}, so we omit most details. \cite[Theorem 5.4]{BD16} implies that there is a unique solution 
${\bf g_2}$ to the differential system \eqref{G2-2}, given by 
$${\bf g_2} = \int_0^t S_{{\bf G}}(t-s)\left[ {\bf A}\DT({\bf h})(s) + \sum_{k=1}^n \nk {\bf A_{b^k}}\DT(\partial^{n-k}{\bff})(s)\right] ds, $$
where $S_{{\bf G}}(t)$ is the semigroup generated by ${\bf G}$ in $L_{x,v}^{\infty}(\langle v \rangle^{\beta}{\bf M}^{-1/2})$; ${\bf A}\DT$ and ${\bf A_{b^k}}\DT$ are vector operators with the $i$-th component $A_i\DT$ and $A_{b^k,i}\DT$ ($1\leq i \leq N$). 
We use the regularising property of ${\bf A}\DT$ operator given in Lemma \ref{L0}, and similarly for ${\bf A_{b^k}}\DT$ due to that 
$\partial_z^k b_{ij}$ follows the same assumption as $b_{ij}$. The exponential decay of ${\bf h}$ and all the lower order $z$-derivatives of ${\bff}$, i.e., $||\partial^k {\bff}|_{L_v^1L_x^{\infty}(\sp)}$ ($0\leq k \leq n-1$) is used, by the assumption for ${\bf h}$ in this proposition and by induction. 

\subsubsection{Step 3: discussion for ${\bf g}$ and final result}
We prove the existence of the solution ${\bf g}$ by an iterative scheme. We start with ${\bf g_1}^{({\bf 0})} = {\bf g_2}^{({\bf 0})} = 0$ and approximate the 
system of equations (\ref{G1})--(\ref{G2}) as follows ($1\leq i \leq N$):
\begin{align*}
\partial_t g_{1,i}^{(p+1)} & = G_{1,i}^{(\delta)}({\bf g_1}^{(p+1)})  + 2 \widetilde Q_i({\bf g_1}^{(p+1)} + {\bf g_2}^{(p)}, {\bff}) \\[4pt]
&\quad + \sum_{k=1}^n \nk \left[ B_{b^k,i}\DT(\partial^{n-k}{\bff}) - \partial^k \nu_i\, \partial^{n-k}f_i \right] + \text{Term}\, \ostar,   \\[6pt]
\partial_t  g_{2,i}^{(p+1)} &= G_i({\bf g_2}^{(p+1)}) + A_i^{(\delta)}({\bf g_1}^{(p+1)}) + \sum_{k=1}^n \nk A_{b^k, i}\DT(\partial^{n-k}{\bff}), 
\end{align*}
with the initial data 
$$ {\bf g_1}^{(p+1)}(0,x,v,z) = {\bf g_0}(x,v,z), \qquad {\bf g_2}^{(p+1)}(0,x,v,z) =0, $$
where ${\bf g_0}\in L_v^1 L_x^{\infty}(\sp)$ satisfies ${\bf \Pi_{G}(g_0)}=0$ for all $z$. 
Recall that ${\bf f_0}\in L_v^1 L_x^{\infty}(\sp)$ with ${\bf \Pi_{G}(f_0)}=0$, since ${\bfg} = \partial_z^n {\bff}$, the initial 
condition for ${\bf g_0}$ holds automatically. 

By Proposition \ref{PG-A} and Proposition \ref{PG-B}, $\left({\bf g_1}^{(p)}\right)_{p\in \mathbb N}$ and 
$\left({\bf g_2}^{(p)}\right)_{p\in \mathbb N}$ are well-defined sequences. 
By induction, we claim that for all $p\in \mathbb N$ and all $t\geq 0$ and each $z \in I_z$, 
\begin{align}
\label{G1-N}\left|\left|{\bf g_1}^{(p)}(t,z)\right|\right|_{L_v^1 L_x^{\infty}(\sp)} & \leq \widetilde C_1\, e^{-\lambda_1 t}, 
\\[4pt]
\label{G2-N}\left|\left|{\bf g_2}^{(p)}(t,z)\right|\right|_{ L_{x,v}^{\infty}( \langle v \rangle^{\beta}{\boldsymbol \mu}^{-1/2})} 
& \leq \widetilde C_2\, e^{-\lambda_2 t}.
\end{align}
If we construct ${\bf g_1}^{(p)}$ and ${\bf g_2}^{(p)}$ satisfying the exponential decay above, then we can obtain ${\bf g_1}^{(p+1)}$ 
from Proposition \ref{PG-A} by letting ${\bf h} = {\bf g_2}^{(p)}$ in equation (\ref{G1-1}) and then construct ${\bf g_2}^{(p+1)}$ with Proposition \ref{PG-B} by letting 
${\bf h} = {\bf g_1}^{(p+1)}$ in equation (\ref{G2-2}). Finally, we have the equality for $1\leq i \leq N$, 
\begin{align*}
 \partial_t \left(g_{1,i}^{(p+1)} + g_{2,i}^{(p+1)}\right) = &\, G_i\left({\bf g_1}^{(p+1)} + {\bf g_2}^{(p+1)}\right) + 
2 \widetilde Q_i({\bf g_1}^{(p+1)} + {\bf g_2}^{(p)}, {\bff}) \\[4pt]
& + \sum_{k=1}^n \nk \left[ A_{b^k, i}\DT(\partial^{n-k}{\bff}) + B_{b^k,i}\DT(\partial^{n-k}{\bff}) - \partial^k \nu_i\, \partial^{n-k}f_i \right] + \text{Term}\, \ostar. 
\end{align*}

In conclusion, for each $z$, $\left({\bf g_1}^{(p)} \right)_{p\in \mathbb N}$ is a Cauchy sequence in $L_t^{\infty}L_v^1 L_x^{\infty}(\sp)$ and converges 
strongly towards a function ${\bf g_1}$. By (\ref{G2-N}), the sequence $\left({\bf g_2}^{(p)}\right)_{p\in \mathbb N}$ is bounded and weakly-$^{\ast}$ converges, up to a subsequence, towards ${\bf g_2}$ in $L_t^{\infty}L_{x,v}^{\infty}(\langle v \rangle^{\beta}{\boldsymbol \mu}^{-1/2})$. 
This implies that $({\bf g_1}, {\bf g_2})$ is solution to the system (\ref{G1})--(\ref{G2}) and ${\bf g} = {\bf g_1} + {\bf g_2}$ is a solution to equation (\ref{G-eqn}) satisfying ${\bf \Pi_{G}(g)}=0$. Moreover, taking the limit inside the exponential decays \eqref{G1-N}  and \eqref{G2-N}, one
concludes that for all $z$, 
$$ ||{\bf g}||_{L_v^1 L_x^{\infty}(\sp)} \leq C e^{-\lambda t}\, ||{\bf g_0}||_{L_v^1 L_x^{\infty}(\sp)}. $$

Recall the notation ${\bfg}=\partial_z^n {\bff}$. We now conclude that 
$$ ||\partial_z^n {\bff}||_{L_v^1 L_x^{\infty}(\sp)L_z^{\infty}} \leq C\, e^{-\lambda t}\, ||\partial_z^n {\bf f_0}||_{L_v^1 L_x^{\infty}(\sp)L_z^{\infty}}, $$
where $C$, $\lambda$ are generic constants that depend on $N$, $k$, collision kernels, initial data of ${\bff}$ and $\partial_z^k {\bff}$ ($1\leq k \leq n$). 

We showed that Proposition \ref{Prop-Main} is true for $m=n$ ($1\leq n \leq r$) by induction, one concludes that 
the result in Proposition \ref{Prop-Main} holds for all $n=0, \cdots, r$, 
where $r$ is associated to the regularity of the initial data ${\bf f_0}$ in the random space.

\section{Spectral gap of the linearized gPC Galerkin system}
\label{sec:4}

In this part, we generalize the single-species gPC-SG system to the multi-species gPC-SG system
by adapting the idea from the proof of the multi-species H-theorem \cite{Des05} and in particular for the Boltzmann model \cite{Linearized-Boltz16}, combined with the previous work considering the uncertainty \cite{LJ18, DJL}. 
We consider in this Section the case of random initial data and random collision kernel, where the distribution of the {\textcolor{black}one-dimensional} random variable $z$ is given by $\pi(z)$. 

The same notation and perturbative setting are followed as that in \cite{Linearized-Boltz16, DJL}. 
Denote $$M_i(v) = \frac{\rho_{\infty,i}}{(2\pi)^{3/2}} e^{- \frac{|v|^2}{2}}, \qquad 1\leq i \leq N. $$
Assume that the distribution function $F_i$ is close to the global equilibrium such that we can write 
\begin{equation}\label{F_Ans}F_i = M_i + M_i^{1/2}f_i, \end{equation}
for some perturbation function $f_i$. 

Plug in the ansatz (\ref{F_Ans}) into (\ref{model}), then $f_i$ satisfies the equation 
\begin{equation}\label{LE}\partial_t f_i + v\cdot\nabla_x f_i = \uwave {L_i}(f) + \uwave{Q_i}(f), \end{equation}
where 
$\uwave {L_i}(f)=\sum_{l=1}^N \uwave {L_{il}}(f_i, f_l)$, with 
\begin{align}
&\displaystyle \uwave{L_{il}}(f_i, f_l)=M_i^{-1/2}\left(Q_{il}(M_i, M_l^{1/2}f_l)+ Q_{il}(M_i^{1/2}f_i, M_l)\right) \notag\\[4pt]
&\displaystyle \label{L_il}\qquad\qquad = \int_{\mathbb R^3\times \mathbb S^2} B_{il}M_i^{1/2}M_l^{\ast}(h_i^{\prime} + h_l^{\prime\ast}
-h_i - h_l^{\ast})\, dv^{\ast}d\sigma, \qquad h_i:=M_i^{-1/2}f_i, 
\end{align}
and $$\uwave{Q_i}(f) = \sum_{l=1}^N M_i^{-1/2}\, Q_{il}(M_i^{1/2}f_i, M_l^{1/2}f_l). $$

It has been shown in \cite{Linearized-Boltz16} that the linearized Boltzmann system (\ref{LE}) satisfies the H-theorem with the linearized entropy
$H(f) = \frac{1}{2}\sum_{i=1}^N \int_{\mathbb R^3} f_i^2\, dv$, that is, 
$$ - \frac{dH}{dt} = -\sum_{i=1}^N \int_{\mathbb R^3} f_i L_i(f)\, dv =: - (f, L(f))_{L^2_v} \geq 0, $$
where $(\cdot, \cdot)_{L^2_v}$ is the scalar product on $L^2_v = L^2(\mathbb R^3; \mathbb R^n)$. 

\begin{remark}
Note that the linearization \eqref{F_Ans} is different from 
\eqref{PS}, with the extra factor $M_i^{1/2}$. The reason is that we will extend the spectral gap analysis from the single-species case studied in \cite{DJL} to the 
multi-species Boltzmann system, thus it is better to follow the same perturbative setting as in \cite{DJL}. 
\end{remark}

One can approximate the distribution for the $i$-th species $f_i$ (or $h_i$) by using the ansatz
\begin{align}
&\displaystyle  f_i(t,x,v,z)\approx f_i^K(t,x,v,z) := {\textcolor{black} \sum_{k=1}^{K}\, f_{i, k}(t,x,v)\psi_k(z)},  \notag\\[2pt]
&\displaystyle \label{gPC-Ans} h_i(t,x,v,z) \approx h_i^K(t,x,v,z) := {\textcolor{black} \sum_{k=1}^{K}\, h_{i,k}(t,x,v)\psi_k(z)}.  
\end{align}

By inserting the ansatz (\ref{gPC-Ans}) into the linearized equation (the linear part of equation (\ref{LE})) 
$$\partial_t f_i + v\cdot\nabla_x f_i = L_i(f), $$
and conducting a standard Galerkin projection, one obtains the following gPC-SG system for $f_{i,k}$ (with $1\leq i \leq N$, $1\leq k \leq K$): 
\begin{equation}\label{gPC}\partial_t f_{i,k} + v\cdot\nabla_x f_{i,k} = \langle L_i(f^K), \psi_k\rangle_{L^2(\pi(z))}. \end{equation}

\renewcommand{\labelenumi}{(B\theenumi)}
In this part of the study for the gPC-Galerkin system, besides (H1)--(H5), we need the following additional assumptions 
{\textcolor{black}(recall that $b_{il}$ is the angular part of the collision kernel $B_{il}$ in \eqref{b-Assump1})}:  
\begin{enumerate}
\item  Assume that $b_{il}$ is linear in $z$, 
          \begin{equation}\label{b-il}  b_{il}(\cos\theta, z) = b_{il}^{(0)}(\cos\theta) + b_{il}^{(1)}(\cos\theta)z. \end{equation}
This assumption is reasonable and a common practice, see the Karhunen-Loeve expansion \cite{Loeve}. 
\item Assume the leading part $b_{il}^{(0)}$ and the perturbative part $b_{il}^{(1)}$ in \eqref{b-il} satisfy the condition
 \begin{equation}\label{b-D} b_{il}^{(0)}(\cos\theta) \geq (2^q + 2)\, |b_{il}^{(1)}(\cos\theta)| \, C_z + D_{il}(\cos\theta), \end{equation}
 where $q$ is associated to the energy $E^K$ defined in \cite{DJL}. 
 
\item The random variable $z$ has a compact support, that is,
 $$|z| \leq C_z. $$
\end{enumerate}

\begin{remark}
We want to mention that due to (B1), our global assumption (H5) has the particular form: 
$$ |b_{il}^{(1)}| \leq C. $$ 
The assumptions (B1)--(B3) are the same as that in \cite{DJL} except now we are in the multi-species framework. 
\end{remark}

The main result of Section \ref{sec:4} is the following theorem: 
\begin{theorem}\label{Thm2} {\bf (Main result of the gPC-Galerkin system)}
Under the assumptions  (H1)--(H5) and (B1)--(B3), and additionally, assume that for all $1\leq i, l \leq N$, $D_{il}(\cos\theta)$ in \eqref{b-D} satisfy the same assumptions as $b(\cos\theta)$ in the deterministic case in \cite{BD16}, then 
we obtain an explicit spectral gap estimate for the linearized operator in the gPC stochastic Galerkin system, in a proper weighted norm, 
$$ \sum_{i=1}^N \sum_{k=1}^K k^{2q} \left\langle\langle L_i(f^K), \psi_k\rangle_{L^2(\pi(z))}, f_{i,k} \right\rangle_{L^2_v} \leq -C \sum_{i=1}^N \sum_{k=1}^K ||k^{2q} f_{i,k}||_{\Lambda}^2, $$ where $C$ is a positive constant independent of $K$, $||\cdot||_{\Lambda}$ is some weighted $L^2_v$ norm. 
\end{theorem}

\subsection{The proof of Theorem \ref{Thm2}}
We denote the right-hand-side of (\ref{gPC}) by $\text{Term}\, \textcircled{a}$, then 
\begin{align*}
&\displaystyle \text{Term}\, \textcircled{a}:= \langle L_i(f^K), \psi_k\rangle_{L^2(\pi(z))} = \langle \sum_{l=1}^N L_{il}(f_i^K, f_l^K), \psi_k \rangle \\[4pt]
&\displaystyle\qquad\qquad = \sum_{l=1}^N \sum_{j=1}^K \int B_{il} M_i^{1/2}M_l^{\ast}\, \psi_k \psi_j \Theta_{il}[h_j]\, dv^{\ast}d\sigma dv \pi(z)dz, 
\end{align*}
where the subscript in $\uwave {L_i}(f)$ is omitted, and we use (\ref{L_il}) and approximate $h_i$ (and $h_l$) by $h_i^K$ (and $h_l^K$) given in (\ref{gPC-Ans}); 
the term $\Theta_{il}[h_j]$ above is denoted by 
$$\Theta_{il}[h_j] := h_{i,j}^{\prime} + h_{l,j}^{\prime\ast} - h_{i,j} - h_{l,j}^{\ast}. $$
For the readers' convenience, we use indices $(i,l)$ to denote different species, while $(j,k)$ stand for the index of the gPC coefficients. 

Take an inner product of $\text{Term}\,\textcircled{a}$ with $f_{i,k}$ on $L^2(v)$, multiply by $k^{2q}$ then sum up $k=1, \cdots, K$ and $i=1, \cdots, N$, we have
\begin{align} 
&\displaystyle\text{Term I}: = \sum_{i=1}^N \sum_{k=1}^K k^{2q} \langle \text{Term}\,\textcircled{a}, f_{i,k}\rangle_{L^2(v)} 
=\sum_{i=1}^N \sum_{k=1}^K k^{2q} \langle \text{Term}\,\textcircled{a}, M_i^{1/2} h_{i,k}\rangle_{L^2(v)} \notag\\[4pt]
&\displaystyle\qquad\quad = \sum_{i=1}^N \sum_{k=1}^K \sum_{l=1}^N \sum_{j=1}^K k^{2q} \int B_{il} M_i M_l^{\ast}\, \psi_k \psi_j 
\Theta_{il}[h_j] h_{i,k}\, d\Omega \notag\\[4pt]
&\displaystyle\label{TermI}\qquad\quad = \sum_{i,l,k,j} \int B_{il} M_i M_l^{\ast}\, \psi_k \psi_j 
\left(h_{i,j}^{\prime} + h_{l,j}^{\prime\ast} - h_{i,j}- h_{l,j}^{\ast}\right) h_{i,k}\, d\Omega, 
\end{align}
where $\sum_{i,l,k,j}:= \sum_{i=1}^N \sum_{l=1}^N \sum_{k=1}^K \sum_{j=1}^K$ and $$d\Omega: = dv^{\ast}d\sigma dv \pi(z)dz$$ are defined
for notational simplicity. 

{\bf Step 1}: Change $(v, v^{\ast})$ to $(v^{\ast}, v)$ in (\ref{TermI}), then exchange $i$ and $l$, one has
\begin{align} 
&\displaystyle\text{Term I} = \sum_{i,l,k,j} k^{2q} \int B_{il} M_i M_l^{\ast}\, \psi_k \psi_j 
\left(h_{i,j}^{\prime\ast} + h_{l,j}^{\prime} - h_{i,j}^{\ast}- h_{l,j}\right) h_{i,k}^{\ast}\, d\Omega \notag\\[4pt]
&\displaystyle\qquad\quad = \sum_{i,l,k,j} k^{2q} \int B_{il} M_i M_l^{\ast}\, \psi_k \psi_j 
\left(h_{i,j}^{\prime} + h_{l,j}^{\prime\ast} - h_{i,j} - h_{l,j}^{\ast}\right) h_{l,k}^{\ast}\, d\Omega, 
\end{align}
where we used $M_i M_l^{\ast} = M_i^{\ast} M_l$ followed by $M_i^{\ast} M_l = M_l^{\ast} M_i$, and $B_{il}=B_{li}$. 

{\bf Step 2}: Change $(v, v^{\ast})$ to $(v^{\prime}, v^{\prime\ast})$ in (\ref{TermI}), one gets
\begin{align} 
&\displaystyle\text{Term I} = \sum_{i,l,k,j} k^{2q} \int B_{il} M_i M_l^{\ast}\, \psi_k \psi_j 
\left(h_{i,j} + h_{l,j}^{\ast} - h_{i,j}^{\prime}- h_{l,j}^{\prime\ast}\right) h_{i,k}^{\prime}\, d\Omega \notag\\[4pt]
&\displaystyle\label{S2}\qquad\quad = - \sum_{i,l,k,j} k^{2q} \int B_{il} M_i M_l^{\ast}\, \psi_k \psi_j 
\left(h_{i,j}^{\prime} + h_{l,j}^{\prime\ast} - h_{i,j}- h_{l,j}^{\ast}\right) h_{i,k}^{\prime}\, d\Omega, 
\end{align}
where we used $M_i^{\prime} M_l^{\prime\ast} = M_i M_l^{\ast}$.

{\bf Step 3}: Change $(v,v^{\ast})$ to $(v^{\ast}, v)$ on (\ref{S2}), then exchange $i$ and $l$, one has
\begin{align} 
&\displaystyle\text{Term I} =  - \sum_{i,l,k,j} k^{2q} \int B_{il} M_i M_l^{\ast}\, \psi_k \psi_j 
\left(h_{i,j}^{\prime\ast} + h_{l,j}^{\prime} - h_{i,j}^{\ast}- h_{l,j}\right) h_{i,k}^{\prime\ast}\, d\Omega \notag\\[4pt]
&\displaystyle\label{S3}\qquad\quad = - \sum_{i,l,k,j} k^{2q} \int B_{il} M_i M_l^{\ast}\, \psi_k \psi_j 
\left(h_{i,j}^{\prime} + h_{l,j}^{\prime\ast} - h_{i,j}- h_{l,j}^{\ast}\right) h_{l,k}^{\prime\ast}\, d\Omega. 
\end{align}
where we used $M_i M_l^{\ast} = M_i^{\ast} M_l$ followed by $M_i^{\ast}M_l = M_l^{\ast} M_i$, and $B_{il}=B_{li}$. 

Adding up equations (\ref{TermI}), (\ref{S2}) and (\ref{S3}), one obtains
\begin{align} 
&\displaystyle\text{Term I} = - \frac{1}{4} \sum_{i,l,k,j} k^{2q} 
\int B_{il} M_i M_l^{\ast}\, \psi_k \psi_j \Theta_{il}[h_j]\Theta_{il}[h_k]\, d\Omega \notag\\[4pt]
&\displaystyle\label{T0}\qquad\quad = - \frac{1}{4} \sum_{i,l,k,j} \left(\frac{k}{j}\right)^q 
\int B_{il} M_i M_l^{\ast}\, \psi_k \psi_j \left(j^q \Theta_{il}[h_j]\right) \left(k^q \Theta_{il}[h_k]\right)\, d\Omega. 
\end{align}

The each index pair $(i,l)$, the above formulation (\ref{T0}) is exactly the same as \cite[equation (39)]{DJL} except now we are in the multispecies setting. 
A similar analysis follows here, and we put it in the Appendix. 
Then in analogous to \cite[equation (44)]{DJL}, one finally obtains that 
\begin{align}
\label{T1}
\begin{split}
&\displaystyle\quad\text{Term I} \\[4pt]
&\displaystyle\leq -\frac{1}{4} \sum_{i,l=1}^N \sum_{k=1}^K \int M_i M_l^{\ast} \Phi_{il}(|v-v^{\ast}|) D_{il}(\cos\theta) \left( k^q \Theta_{il}[h_k] \right)^2 
dv^{\ast}d\sigma dv \\[4pt]
&\displaystyle\leq -\frac{1}{4} \sum_{i,l=1}^N \sum_{k=1}^K k^{2q}\int M_i M_l^{\ast} \Phi_{il}(|v-v^{\ast}|) D_{il}(\cos\theta) \left( \Theta_{il}[h_k] \right)^2 
dv^{\ast}d\sigma dv \\[4pt]
&\displaystyle = \sum_{i,l=1}^N \sum_{k=1}^K k^{2q} \int M_i M_l^{\ast} \Phi_{il}(|v-v^{\ast}|) D_{il}(\cos\theta) \Theta_{il}[h_k] h_{i,k}\, dv^{\ast}d\sigma dv
\\[4pt]
&\displaystyle = \sum_{i=1}^N \sum_{k=1}^K k^{2q} \langle L_i^{\widetilde D}(f_k), f_{i,k} \rangle, 
\end{split}
\end{align}
where we define $h_i = M_i^{-1/2}f_i$ and 
\begin{align}
\label{L-D}
\begin{split}
&\displaystyle L_i^{\tilde D}(f_k) := \sum_{l=1}^N \int \widetilde D_{il}(|v-v^{\ast}|, \cos\theta) M_i^{1/2} M_l^{\ast} (h_{i,k}^{\prime} + h_{l,k}^{\prime\ast} - h_{i,k} - h_{l,k}^{\ast})\, dv^{\ast}d\sigma dv,  \\[4pt]
&\displaystyle \widetilde D_{il}(|v-v^{\ast}|, \cos\theta):= \Phi_{il}(|v-v^{\ast}|) D_{il}(\cos\theta), 
\end{split}
\end{align}
Integrating on $x$ of \eqref{T1}, we finally get 
$$ \text{Term I} \leq \sum_{i=1}^N \sum_{k=1}^K k^{2q} \langle L_i^{\widetilde D}(f_k), f_{i,k} \rangle_{L^2_v} \leq -C \sum_{i=1}^N \sum_{k=1}^K ||k^{2q} f_{i,k}||_{\Lambda}^2. $$

The proof of Theorem \ref{Thm2} is done. We generalized the spectral gap proof for the linearized numerical collision operator of the single-species Boltzmann equation studied in \cite{DJL} to the multi-species setting, which will be prepared for 
studying the long-time behavior and spectral convergence for the numerical solution (and numerical error) for the gPC Galerkin system, 
as done for the analytical solution in Section \ref{sec:3}. 
{\textcolor{black}We mention that in \cite{LJ18}, hypocoercivity of the SG system and regularity of its solution in a weighted Sobolev norm, 
as well as spectral accuracy and exponential decay in time of the numerical error of the gPC-SG method has been established.}
In \cite{Linearized-Boltz16}, the authors have studied the convergence to equilibrium in $H_{x,v}^1$ space for the linearized multi-species Boltzmann equations, 
{\textcolor{black}nevertheless} the study of convergence to equilibrium in higher Sobolev space $H_{x,v}^s$ for the nonlinear deterministic equations is not yet developed, 
so a complete above-mentioned study in the uncertainty framework for the gPC Galerkin system remains a future work. 

\section{Conclusion}
\label{sec:5}
In this paper, we consider the nonlinear multi-species Boltzmann equation with uncertainty coming from both the 
initial data and collision kernels. Well-posedness and regularity in the random space of the solution to the sensitivity system --
the PDE obtained from taking derivatives in the random space, 
long-time behavior (exponential decay to the global equilibrium) of the analytic solution,
spectral gap of the linearized corresponding gPC-based stochastic Galerkin system are established.

\section*{Acknowledgement}
All three authors acknowledge the financial support from the Hausdorff Research Institute for Mathematics at University of Bonn in Germany, 
in the framework of the Junior Trimester Program ``Kinetic Theory", and would like to thank the institute for the great hospitality during their stay in the trimester program. 

\begin{appendices}

\section{Proof of Lemma \ref{Q-I}}
\noindent{\bf Proof of \eqref{Q-eqn}:}\, 
If $n$ is odd, one has
\begin{align*}
&\displaystyle\quad \sum_{j=1}^N \sum_{k=1}^{n-1}\nk Q_{ij}(\partial^k f_i, \partial^{n-k}f_j)  \\[2pt]
&= \sum_{k=1}^{\frac{n-1}{2}} \nk \sum_{j=1}^N Q_{ij}(\partial^k f_i, \partial^{n-k}f_j)
+ \sum_{k=\frac{n+1}{2}}^n \nk \sum_{j=1}^N Q_{ij}(\partial^k f_i, \partial^{n-k}f_j) \\[2pt]
&\displaystyle = \sum_{k=1}^{\frac{n-1}{2}} \nk \sum_{j=1}^N Q_{ij}(\partial^k f_i, \partial^{n-k}f_j)
+ \sum_{k^{\prime}=1}^{\frac{n-1}{2}}\binom{n}{n-k^{\prime}} \sum_{j=1}^N Q_{ij}(\partial^{n-k^{\prime}}f_i, \partial^{k^{\prime}} f_j) \\[2pt]
&= \sum_{k=1}^{\frac{n-1}{2}} \sum_{j=1}^N \left[ Q_{ij}(\partial^k f_i, \partial^{n-k}f_j) + Q_{ij}(\partial^{n-k}f_i, \partial^k f_j)\right] \\[2pt]
&= 2 \sum_{k=1}^{\frac{n-1}{2}} \widetilde Q_i (\partial^k {\bff}, \partial^{n-k}{\bff}), 
\end{align*}
where we used the change of variable $k^{\prime}=n-k$ and $\displaystyle\nk = \binom{n}{n-k}$ in the second and third equalities. 

If $n$ is even, similarly one has 
\begin{align*}
&\displaystyle\quad \sum_{j=1}^N \sum_{k=1}^{n-1}\nk Q_{ij}(\partial^k f_i, \partial^{n-k}f_j)  \\[2pt]
&\displaystyle =  \sum_{k=1}^{\frac n 2 -1} \nk \sum_{j=1}^N Q_{ij}(\partial^k f_i, \partial^{n-k}f_j)
+ \sum_{k=\frac n 2 +1}^{n-1}\nk \sum_{j=1}^N Q_{ij}(\partial^k f_i, \partial^{n-k}f_j) + 
\binom{n}{\frac n 2}\sum_{j=1}^N Q_{ij}(\partial^{\frac n 2}f_i, \partial^{\frac n 2}f_j)  \\[2pt]
&\displaystyle = \sum_{k=1}^{\frac n 2 -1} \nk \sum_{j=1}^N \left[ Q_{ij}(\partial^k f_i, \partial^{n-k}f_j) + Q_{ij}(\partial^{n-k}f_i, \partial^k f_j)\right] 
+ \binom{n}{\frac n 2}\sum_{j=1}^N Q_{ij}(\partial^{\frac n 2}f_i, \partial^{\frac n 2}f_j)  \\[2pt]
&\displaystyle = 2 \sum_{k=1}^{\frac n 2 -1} \nk \widetilde Q_i(\partial^k f_i, \partial^{n-k}f_j) 
+ \binom{n}{\frac n 2} \widetilde Q_i (\partial^{\frac n 2}{\bff}, \partial^{\frac n 2}{\bff}). 
\end{align*}
Combine the two cases, then 
$$  \sum_{j=1}^N \sum_{k=1}^{n-1}\nk Q_{ij}(\partial^k f_i, \partial^{n-k}f_j) = 
2 \sum_{k=1}^{\lfloor\frac{n-1}{2}\rfloor}\nk \widetilde Q_i (\partial^k {\bff}, \partial^{n-k}{\bff}) 
+ \chi_n\binom{n}{\frac n 2}\widetilde Q_i (\partial^{\frac n 2}{\bff},\partial^{\frac n 2}{\bff}). $$
\eqref{Q-eqn} is proved. 
\\[8pt]

\noindent{\bf Proof of \eqref{Q-eqn2}:}\, 
We recall \cite[Proof of Lemma 6.6]{BD16}, the difference is that here $Q_{ij}^{b^{\ell}}$ involves the $z$-derivatives of the collision kernel 
$B$: 
$$ Q_{ij}^{b^{\ell}}(f_i, g_j) = \int_{\mathbb R^3 \times \mathbb S^2} \partial_z^{\ell}B_{ij}\left( f_i^{\prime}g_j^{\prime\ast} 
- f_i g_j^{\ast}\right) dv^{\ast} d\sigma. $$
By Minkowski's integral inequality, for all $q \in [1, \infty)$, 
\begin{align*}
\int_{\mathbb{R}^{3}}\langle v\rangle^{k}\left[\int_{\mathbb{T}^{3}}\left|Q_{i j}^{b^{\ell}}\left(f_{i}, g_{j}\right)\right|^{q} d x\right]^{1/q} dv  & \leqslant \int_{\mathbf{S}^{2} \times \mathbb{R}^{3} \times \mathbb{R}^{3}}\langle v\rangle^{k}\left[\int_{\mathbb{T}^{3}}\left| \partial_z^{\ell} B_{i j}\, f_{i}^{\prime} g_{j}^{\prime\ast}\right|^{q} d x\right]^{1 / q} d \sigma d v^{\ast} d v \\[2pt]
 &+\int_{S^{2} \times \mathbb{R}^{3} \times \mathbb{R}^{3}}\langle v\rangle^{k}\left[\int_{\mathbb{T}^{3}}\left| \partial_z^{\ell} B_{i j}\, f_{i} g_{j}^{\ast}\right|^{q} d x\right]^{1 / q} d \sigma d v^{\ast} d v. 
\end{align*}

We make the change of variables $(v, v^{\ast})\to (v^{\prime}, v^{\prime\ast})$ in the first integral and obtain
\begin{align*}
&\quad\int_{\mathbb{R}^{3}}\langle v\rangle^{k}\left[\int_{\mathbb{T}^{3}}\left|Q_{i j}^{b^{\ell}}\left(f_{i}, g_{j}\right)\right|^{q} d x\right]^{1 / q} d v \\[2pt]
&\leqslant \int_{\mathbb{S}^{2} \times \mathbb{R}^{3} \times \mathbb{R}^{3}}\left(\langle v^{\prime}\rangle^{k}+\left\langle v\right\rangle^{k}\right)\left[\int_{\mathbb{T}^{3}}\left| \partial_z^{\ell} B_{i j}\, f_{i} g_{j}^{\ast}\right|^{q} d x\right]^{1/q} d \sigma d v^{\ast} d v \\[2pt]
& \leqslant C_{i j} \int_{\mathbb{S}^{2} \times \mathbb{R}^{3} \times \mathbb{R}^{3}}\langle v\rangle^{k}\left\langle v^{\ast}\right\rangle^{k}\left|v-v^{\ast}\right|^{\gamma}\left[\int_{\mathbb{T}^{3}}\left|f_{i} g_{j}^{\ast}\right|^{q} d x\right]^{1/q} d \sigma d v^{\ast} d v, 
\end{align*}
where the boundness of $|\partial_z^{\ell} b_{ij}|$ is used, and $C_{ij}$ is a constant. 
Finally, by using $\left|v-v^{\ast}\right|^{\gamma} \leqslant\langle v\rangle^{\gamma}+\left\langle v^{\ast}\right\rangle^{\gamma}$ for 
$\gamma\in [0,1]$, one has
\begin{align*}
&\quad \int_{\mathbb{R}^{3}}\langle v\rangle^{k} \left[\int_{\mathbb{T}^{3}}\left|Q_{i j}^{b^{\ell}}\left(f_{i}, g_{j}\right)\right|^{q} d x\right]^{1 / q} d v \\[2pt]
 & \leqslant C_{i j} \int_{\mathbb{S}^{2} \times \mathbb{R}^{3} \times \mathbb{R}^{3}}\left(\langle v\rangle^{k+\gamma}\left\langle v^{*}\right\rangle^{k}+\langle v\rangle^{k}\left\langle v^{*}\right\rangle^{k+\gamma}\right)\left[\int_{\mathbb{T}^{3}}\left|f_{i} g_{j}^{*}\right|^{q} d x\right]^{1 / q} d \sigma d v^{*} d v. 
 \end{align*}
Take the limit as $q$ tends to infinity, then
\begin{align*}
\left\|Q_{i j}^{b^{\ell}}\left(f_{i}, g_{j}\right)\right\|_{L_{v}^{1} L_{x}^{\infty}\left(\langle v\rangle^{k}\right)} 
&\leqslant C_{i j}\left[\left\|f_{i}\right\|_{L_{v}^{1} L_{x}^{\infty}\left(\langle v\rangle^{k}\right)} \left\|g_{j}\right\|_{L_{v}^{1} L_{x}^{\infty} \left(\langle v\rangle^{k+\gamma}\right)}\right. \\[4pt]
&\qquad\quad + \left. \left\|f_{i}\right\|_{L_{v}^{1} L_{x}^{\infty}\left(\langle v\rangle^{k+\gamma}\right)}\left\|g_{j}\right\|_{L_{v}^{1} L_{x}^{\infty}\left(\langle v\rangle^{k}\right)}\right]. 
\end{align*}
Summing over $j$, let $\widetilde C_Q$ be the maximum of all $C_{ij}$, one obtains 
$$ \sum_{j=1}^N  || Q_{ij}^{b^{\ell}}(f_i,  g_j)||_{L_v^1 L_x^{\infty}(\sp)}
\leq \widetilde C_Q \left[ ||f_i||_{L_v^1 L_x^{\infty}(\sp)} ||{\bfg}||_{L_v^1 L_x^{\infty}(\langle v \rangle^{k+\gamma})} + ||f_i||_{L_v^1 L_x^{\infty}\left(\langle v\rangle^{k+\gamma}\right)} ||{\bf g}||_{L_v^1 L_x^{\infty}(\sp)}\right]. $$
Consequently, since $\nu_i \sim \langle v \rangle^{\gamma}$, then
$$ \sum_{j=1}^N  || Q_{ij}^{b^{\ell}}(f_i,  g_j)||_{L_v^1 L_x^{\infty}(\sp)}
\leq \widetilde C_Q \left[ ||f_i||_{L_v^1 L_x^{\infty}(\sp)} ||{\bfg}||_{L_v^1 L_x^{\infty}(\sp \boldsymbol{\nu})} + ||f_i||_{L_v^1 L_x^{\infty}(\nu_i\sp)} ||{\bf g}||_{L_v^1 L_x^{\infty}(\sp)}\right]. $$

\hspace{10cm}

\section{Derivation from \eqref{T0} to \eqref{T1}}
This part is similar to \cite{DJL} but in the multispecies setting. 
Define the integral 
$$ S_{i,l,k,j} = \int_{I_z} B_{il} \psi_k(z)\psi_j(z) \pi(z)dz, \qquad 1\leq i, l \leq N, \, 1\leq k, j \leq K. $$
Denote $d\xi = dv^{\ast}d\sigma dv$, and $$ \widetilde\Theta_{il}[h_j] = j^q \,\Theta_{il}[h_j], \qquad 1\leq i, l \leq N, \, 1\leq j \leq K. $$
Then from \eqref{T0},
 \begin{equation}\label{TT1}\text{Term I} = -\frac{1}{4}\sum_{i,l,k,j}\left(\frac{k}{j}\right)^q \, \int M_i M_l^{\ast}\, S_{i,l,k,j}\, \widetilde\Theta_{il}[h_j]\, \widetilde\Theta_{il}[h_k]\,
d\xi. \end{equation}
Define $\widetilde S_{i,l,k,j}$ by $S_{i,l,k,j} = \Phi_{i,l}\, \widetilde S_{i,l,k,j}$. 
{\textcolor{black}By assumption (B1), we let}
 $b_{i,l} = b_{i,l}^{(0)} + b_{i,l}^{(1)}z$, then 
\begin{equation}\label{S-il} \widetilde S_{i,l,k,j} = b_{i,l}^{(0)}\delta_{kj} + b_{i,l}^{(1)}\int_{I_z} z \psi_k \psi_j\, d\pi(z). \end{equation}
We focus on calculating the summation: 
$$\text{Term A}:=\sum_{i,l,k,j} \left(\frac{k}{j}\right)^q\, M_i M_l^{\ast}\, S_{i,l,k,j}\,  \widetilde\Theta_{il}[h_j]\, \widetilde\Theta_{il}[h_k]. $$
Plug in the form \eqref{S-il}, then 
\begin{align*}
\displaystyle \text{Term A} &= \sum_{i,l} M_i M_l^{\ast}\, b_{i,l}^{(0)}\, 
\sum_{k,j} \left(\frac{k}{j}\right)^{q}\, \widetilde\Theta_{i,l}[h_j]\, \widetilde\Theta_{i,l}[h_k]\, \delta_{kj} \\[4pt]
&\displaystyle\quad + \sum_{i,l} M_i M_l^{\ast} b_{i,l}^{(1)}\, \sum_{k,j}
 \left(\frac{k}{j}\right)^{q}\, \widetilde\Theta_{i,l}[h_j]\, \widetilde\Theta_{i,l}[h_k]\, \int_{I_z}z\, \psi_k\psi_j\, d\pi(z) \\[4pt]
&\displaystyle=  \sum_{i,l} M_i M_l^{\ast}\, b_{i,l}^{(0)}\, \sum_k \widetilde\Theta_{i,l}^2[h_k]  +  \text{Term B}. 
\end{align*}
Notice that Term B is non-zero only when $j=k-1$, $j=k$ or $j=k+1$ due to the integral $\int_{I_z}z\, \psi_k\psi_j\, d\pi(z)$. Thus 
\footnotesize{
\begin{align*}
\displaystyle |\text{Term B}| &\leq
 \sum_{i,l} M_i M_l^{\ast} \Bigg\{ |b_{i,l}^{(1)}| \sum_{k=2}^K 
 \left| \widetilde\Theta_{i,l}[h_{k}]\,  \widetilde\Theta_{i,l}[h_{k-1}]\left(\frac{k}{k-1}\right)^{q} \int_{I_z}z\, \psi_k\psi_{k-1}\, d\pi(z)\right| \\[4pt]
 &\displaystyle\quad + |b_{i,l}^{(1)}| \sum_{k=1}^{K-1} 
 \left| \widetilde\Theta_{i,l}[h_{k}]\,  \widetilde\Theta_{i,l}[h_{k+1}]\left(\frac{k}{k+1}\right)^{q} \int_{I_z}z\, \psi_k\psi_{k+1}\, d\pi(z)\right| 
 + |b_{i,l}^{(1)}| \sum_{k=1}^K  \left| \widetilde\Theta_{i,l}^2[h_{k}] \int_{I_z}z\, \psi_k^2\, d\pi(z)\right| \Bigg\} \\[4pt]
&\displaystyle \leq \sum_{i,l} M_i M_l^{\ast} \Bigg\{ 2^q\, |b_{i,l}^{(1)}| \sum_{k=2}^{K} \left| \widetilde\Theta_{i,l}[h_{k}]\, \widetilde\Theta_{i,l}[h_{k-1}]\right| 
\left|\int_{I_z} z\, \psi_k \psi_{k-1}\, d\pi(z)\right|   \\[4pt]
&\displaystyle\quad +  |b_{i,l}^{(1)}|\sum_{k=1}^{K-1}\left| \widetilde\Theta_{i,l}[h_{k}]\, \widetilde\Theta_{i,l}[h_{k+1}]\right|
 \left|\int_{I_z}z\, \psi_k \psi_{k+1}\, d\pi(z)\right| 
+  |b_{i,l}^{(1)}|\sum_{k=1}^{K}\widetilde\Theta_{i,l}^2[h_{k}]\,\left|\int_{I_z} z\, \psi_k^2\, d\pi(z)\right| \Bigg\} \\[4pt]
&\displaystyle \leq \sum_{i,l} M_i M_l^{\ast} \Bigg\{ 2^q\, |b_{i,l}^{(1)}|\, C_z \sum_{k=2}^{K} \left| \widetilde\Theta_{i,l}[h_{k}]\,  \widetilde\Theta_{i,l}[h_{k-1}]\right|
+ |b_{i,l}^{(1)}|\, C_z \sum_{k=1}^{K-1} \left|\widetilde\Theta_{i,l}[h_{k}]\,  \widetilde\Theta_{i,l}[h_{k+1}]\right|
+ |b_{i,l}^{(1)}|\, C_z \sum_{k=1}^{K}\widetilde\Theta_{i,l}^2[h_{k}]  \Bigg\}\\[4pt]
&\displaystyle\leq  \sum_{i,l} M_i M_l^{\ast} \Bigg\{ 2^q\,  |b_{i,l}^{(1)}|\, C_z\, \frac{1}{2}\left(\sum_{k=2}^K  \widetilde\Theta_{i,l}^2[h_{k}]
+ \widetilde\Theta_{i,l}^2[h_{k-1}]\right) +  |b_{i,l}^{(1)}|\, C_z\, \frac{1}{2}\left(\sum_{k=1}^{K-1}\widetilde\Theta_{i,l}^2[h_{k}] +  
\sum_{k=1}^{K-1} \widetilde\Theta_{i,l}^2[h_{k+1}]\right) \\[4pt]
&\displaystyle\qquad\quad +  |b_{i,l}^{(1)}|\, C_z \sum_{k=1}^{K}\widetilde\Theta_{i,l}^2[h_{k}] \Bigg\}  \\[4pt]
&\displaystyle \leq  \sum_{i,l} M_i M_l^{\ast}\, (2^q+2)\, |b_{i,l}^{(1)}|\, C_z \sum_{k=1}^K \widetilde\Theta_{i,l}^2[h_{k}]\,, 
\end{align*}
}
where we used that, {\textcolor{black}due to assumption (B3),}
\begin{align*}
\left|\int_{I_{z}} z \psi_{k} \psi_{k-1} d \pi(z)\right| & \leq\|z\|_{L^{\infty}} \int_{I_{z}}\left|\psi_{k} \psi_{k-1}\right| d \pi(z) \\ & \leq C_{z}\left(\int_{I_{z}} \psi_{k}^{2} d \pi(z)\right)^{1 / 2}\left(\int_{I_{z}} \psi_{k-1}^{2} d \pi(z)\right)^{1 / 2}=C_{z}. 
\end{align*}
Therefore, 
\begin{align*} 
\text { Term } \mathrm{A} & \geq \sum_{i,l} M_i M_l^{\ast} b_{i,l}^{(0)}
 \sum_{k=1}^{K} \widetilde\Theta_{i,l}^2[h_{k}]  - \sum_{i,l} M_i M_l^{\ast}
 \left(2^{q}+2\right) |b_{i,l}^{(1)}|\, C_{z}\, \sum_{k=1}^{K} \widetilde\Theta_{i,l}^2[h_{k}]  \\[4pt]
 &=\left(b_{i,l}^{(0)}-\left(2^{q}+2\right)  |b_{i,l}^{(1)}|\, C_{z}\right)  \sum_{k=1}^{K} \widetilde\Theta_{i,l}^2[h_{k}] 
 \geq  \sum_{i,l} M_i M_l^{\ast}\, D_{il}(\cos \theta)\, \sum_{k=1}^{K}\widetilde\Theta_{i,l}^2[h_{k}]. 
  \end{align*}
{\textcolor{black} Note that the assumption (B2) is used in the last inequality. }

By \eqref{TT1}, one finally obtains that 
\begin{align*}
 \text{Term I} & 
 = - \frac{1}{4} \sum_{i,l,k,j} \left(\frac{k}{j}\right)^q\, 
 \int M_i M_l^{\ast}\, \Phi_{il}(|v-v^{\ast}|)\, \widetilde S_{i,l,k,j}(\cos\theta)\, \widetilde\Theta_{il}[h_j]\, 
 \widetilde\Theta_{il}[h_k]\, d\xi \\[4pt]
 & \leq - \frac{1}{4} \sum_{i,l} \int M_i M_l^{\ast}\, \Phi_{il}(|v-v^{\ast}|)\, D_{il}(\cos \theta)\,
 \sum_{k=1}^{K} k^{2q}\, \Theta_{il}^2[h_k]\, dv^{\ast}d\sigma dv. 
 \end{align*}
 We finish the derivation from \eqref{T0} to \eqref{T1}. 
\end{appendices}

\bibliographystyle{siam}
\bibliography{Multi_Boltzmann.bib}

\begin{thebibliography}{10}

\bibitem{BBBD}
{\sc C.~Baranger, M.~Bisi, S.~Brull, and L.~Desvillettes}, {\em On the
  {C}hapman-{E}nskog asymptotics for a mixture of monoatomic and polyatomic
  rarefied gases}, Kinet. Relat. Models, 11 (2018), pp.~821--858.

\bibitem{BBD18}
\leavevmode\vrule height 2pt depth -1.6pt width 23pt, {\em On the
  {C}hapman-{E}nskog asymptotics for a mixture of monoatomic and polyatomic
  rarefied gases}, Kinet. Relat. Models, 11 (2018), pp.~821--858.

\bibitem{BBBG}
{\sc A.~Bondesan, L.~Boudin, M.~Briant, and B.~Grec}, {\em Stability of the
  spectral gap for the {B}oltzmann multi-species operator linearized around
  non-equilibrium {M}axwell distribution}, accepted for publication in Commun.
  Pure Appl. Anal., arxiv: 1811.08350,  (2019).

\bibitem{BGP}
{\sc L.~Boudin, B.~Grec, and V.~Pavan}, {\em The {M}axwell-{S}tefan diffusion
  limit for a kinetic model of mixtures with general cross sections}, Nonlinear
  Anal., 159 (2017), pp.~40--61.

\bibitem{BGPS}
{\sc L.~Boudin, B.~Grec, M.~Pavi\'{c}, and F.~Salvarani}, {\em Diffusion
  asymptotics of a kinetic model for gaseous mixtures}, Kinet. Relat. Models, 6
  (2013), pp.~137--157.

\bibitem{BGPS13}
\leavevmode\vrule height 2pt depth -1.6pt width 23pt, {\em Diffusion
  asymptotics of a kinetic model for gaseous mixtures}, Kinet. Relat. Models, 6
  (2013), pp.~137--157.

\bibitem{BS}
{\sc L.~Boudin and F.~Salvarani}, {\em Compactness of linearized kinetic
  operators}, in From particle systems to partial differential equations.
  {III}, vol.~162 of Springer Proc. Math. Stat., Springer, [Cham], 2016,
  pp.~73--97.

\bibitem{Marc16}
{\sc M.~Briant}, {\em Stability of global equilibrium for the multi-species
  {B}oltzmann equation in {$L^\infty$} settings}, Discrete Contin. Dyn. Syst.,
  36 (2016), pp.~6669--6688.

\bibitem{Marc}
\leavevmode\vrule height 2pt depth -1.6pt width 23pt, {\em Perturbative theory
  for the {B}oltzmann equation in bounded domains with different boundary
  conditions}, Kinet. Relat. Models, 10 (2017), pp.~329--371.

\bibitem{BD16}
{\sc M.~Briant and E.~S. Daus}, {\em The {B}oltzmann equation for a
  multi-species mixture close to global equilibrium}, Arch. Ration. Mech.
  Anal., 222 (2016), pp.~1367--1443.

\bibitem{Cercignani}
{\sc C.~Cercignani}, {\em Rarefied gas dynamics: From basic concepts to actual
  calculations}, Cambridge Texts in Applied Mathematics, Cambridge University
  Press, Cambridge,  (2000), pp.~xviii+320.

\bibitem{ZLM-DG}
{\sc Z.~Chen, L.~Liu, and L.~Mu}, {\em D{G}-{IMEX} stochastic {G}alerkin
  schemes for linear transport equation with random inputs and diffusive
  scalings}, J. Sci. Comput., 73 (2017), pp.~566--592.

\bibitem{DJL}
{\sc E.~S. Daus, S.~Jin, and L.~Liu}, {\em Spectral convergence of the
  stochastic {G}alerkin approximation to the {B}oltzmann equation with multiple
  scales and large random perturbation in the collision kernel}, Kinet. Relat.
  Models, 12 (2019), pp.~909--922.

\bibitem{Linearized-Boltz16}
{\sc E.~S. Daus, A.~J\"{u}ngel, C.~Mouhot, and N.~Zamponi}, {\em Hypocoercivity
  for a linearized multispecies {B}oltzmann system}, SIAM J. Math. Anal., 48
  (2016), pp.~538--568.

\bibitem{DesPer}
{\sc B.~Despr\'es and B.~Perthame}, {\em Uncertainty propagation; intrusive
  kinetic formulations of scalar conservation laws}, SIAM/ASA J. Uncertain.
  Quantif., 4 (2016), pp.~980--1013.

\bibitem{Des05}
{\sc L.~Desvillettes, R.~Monaco, and F.~Salvarani}, {\em A kinetic model
  allowing to obtain the energy law of polytropic gases in the presence of
  chemical reactions}, Eur. J. Mech. B Fluids, 24 (2005), pp.~219--236.

\bibitem{DPZ}
{\sc G.~Dimarco, L.~Pareschi, and M.~Zanella}, {\em Uncertainty quantification
  for kinetic models in socio-economic and life sciences}, in Uncertainty
  quantification for hyperbolic and kinetic equations, vol.~14 of SEMA SIMAI
  Springer Ser., Springer, Cham, 2017, pp.~151--191.

\bibitem{Gamba19}
{\sc I.~M. Gamba and M.~Pavi{\'{c}}-{\v{C}}oli{\'{c}}}, {\em On existence and
  uniqueness to homogeneous {B}oltzmann flows of monatomic gas mixtures},
  Archive for Rational Mechanics and Analysis,  (2019).

\bibitem{Ghanem}
{\sc R.~G. Ghanem and P.~D. Spanos}, {\em Stochastic finite elements: A
  spectral approach}, Springer-Verlag, New York,  (1991), pp.~x+214.

\bibitem{Multi-99}
{\sc V.~Giovangigli}, {\em Multicomponent flow modeling}, Modeling and
  Simulation in Science, Engineering and Technology, Birkh\"auser Boston, Inc.,
  Boston, MA (1999), pp.~xvi+321.

\bibitem{GMM}
{\sc M.~P. Gualdani, S.~Mischler, and C.~Mouhot}, {\em Factorization of
  non-symmetric operators and exponential {H}-theorem}, M\'em. Soc. Math. Fr.
  (N.S.), 153 (2017).

\bibitem{GWZ}
{\sc M.~D. Gunzburger, C.~G. Webster, and G.~Zhang}, {\em Stochastic finite
  element methods for partial differential equations with random input data},
  Acta Numerica, 23 (2014), pp.~521--650.

\bibitem{HuReview}
{\sc J.~Hu and S.~Jin}, {\em Uncertainty quantification for kinetic equations},
  Uncertainty Quantification for Kinetic and Hyperbolic Equations, SEMA-SIMAI
  Springer Series, ed. S. Jin and L. Pareschi,  (2017), pp.~193--229.

\bibitem{Ma}
{\sc S.~Jin, J.-G. Liu, and Z.~Ma}, {\em Uniform spectral convergence of the
  stochastic {G}alerkin method for the linear transport equations with random
  inputs in diffusive regime and a micro--macro decomposition-based
  asymptotic-preserving method}, Res. Math. Sci., 4 (2017), pp.~1--25.

\bibitem{Jin-Liu}
{\sc S.~Jin and L.~Liu}, {\em An asymptotic-preserving stochastic {G}alerkin
  method for the semiconductor {B}oltzmann equation with random inputs and
  diffusive scalings}, Multiscale Model. Simul., 15 (2017), pp.~157--183.

\bibitem{UQ-Book}
{\sc S.~Jin and L.~Pareschi}, eds., {\em Uncertainty quantification for
  hyperbolic and kinetic equations}, vol.~14 of SEMA SIMAI Springer Series,
  Springer, Cham, 2017.

\bibitem{Qin-Li}
{\sc Q.~Li and L.~Wang}, {\em Uniform regularity for linear kinetic equations
  with random input based on hypocoercivity}, SIAM/ASA J. Uncertain. Quantif.,
  5 (2017), pp.~1193--1219.

\bibitem{Liu-BP}
{\sc L.~Liu}, {\em A stochastic asymptotic-preserving scheme for the bipolar
  semiconductor {B}oltzmann-{P}oisson system with random inputs and diffusive
  scalings}, J. Comput. Phys., 376 (2019), pp.~634--659.

\bibitem{LJ18}
{\sc L.~Liu and S.~Jin}, {\em Hypocoercivity based sensitivity analysis and
  spectral convergence of the stochastic {G}alerkin approximation to
  collisional kinetic equations with multiple scales and random inputs},
  Multiscale Model. Simul., 16 (2018), pp.~1085--1114.

\bibitem{LM19}
{\sc L.~Liu and M.~Pirner}, {\em Hypocoercivity for a {BGK} model for gas
  mixtures}, J. Differential Equations, 267 (2019), pp.~119--149.

\bibitem{Loeve}
{\sc M.~Lo\`eve}, {\em {Probability Theory I}}, Springer-Verlag New York,
  (1977).

\bibitem{Mouhot1}
{\sc C.~Mouhot}, {\em Explicit coercivity estimates for the linearized
  {B}oltzmann and {L}andau operators}, Comm. Partial Differential Equations, 31
  (2006), pp.~1321--1348.

\bibitem{SmithBook}
{\sc R.~C. Smith}, {\em Uncertainty {Q}uantification: {T}heory,
  {I}mplementation, and {A}pplications}, 12 (2014), pp.~XVIII+382.

\bibitem{Xiu}
{\sc D.~Xiu}, {\em Numerical methods for stochastic computations: {A} spectral
  method approach}, Princeton University Press, Princeton, New Jersey,  (2010).

\end{thebibliography}

\end{document}